\documentclass[twoside,11pt]{article}

\usepackage[wby]{callouts}
\usepackage{graphicx}
\usepackage{amsmath}
\usepackage{amssymb}
\usepackage{amsthm}
\usepackage[cmtip,all]{xy}
\usepackage{braket,amsfonts}
\usepackage{array}
\usepackage[caption=false]{subfig}
\usepackage{pgfplots}
\usepackage{algorithmic}
\usepackage{graphics}
\usepackage{psfrag}
\usepackage{color}
\usepackage{epsfig}
\usepackage{cite}
\usepackage{cleveref}
\usepackage{float}
\usepackage{wrapfig}
\usepackage{bbold}
\usepackage[margin=1.2in]{geometry}

\newtheorem{theorem}{Theorem}
\newtheorem{lemma}{Lemma}

\newtheorem{definition}{Definition}

\newtheorem{assumption}{Assumption}
\newtheorem{example}{Example}

\newcommand{\subscr}[2]{#1_{\textup{#2}}}

\newcommand{\real}{{\mathbb{R}}}

\newcommand{\realnonnegative}{{\mathbb{R}}_{\ge 0}}

\newcommand{\naturals}{\mathbb{N}}

\newcommand{\Xsp}{\mathbb{X}}
\newcommand{\Wsp}{\mathbb{W}}
\newcommand{\Ysp}{\mathbb{Y}}
\newcommand{\Vsp}{\mathbb{V}}
\newcommand{\yvec}{\mathbf{y}}

\newcommand{\dvol}{\operatorname{dvol}}
\newcommand{\vol}{\operatorname{vol}}
\newcommand{\diff}{\operatorname{d}}
\newcommand{\hess}{\operatorname{Hess}}
\newcommand{\supp}{\operatorname{supp}}
\newcommand{\RMSE}{\operatorname{RMSE}}

\newcommand{\longthmtitle}[1]{\mbox{}{\bf
      \textit{(#1).}}}

\newcommand\oprocendsymbol{\hbox{$\bullet$}}
\newcommand\oprocend{\relax\ifmmode\else\unskip\hfill\fi\oprocendsymbol}

\title{\LARGE \bf A Probabilistic Framework for Moving-Horizon
  Estimation: Stability and Privacy Guarantees}
\author{Vishaal Krishnan \and Sonia Mart{\'\i}nez
\thanks{The authors
    are with the Department of Mechanical and Aerospace Engineering,
    University of California at San Diego, La Jolla CA 92093 USA
    (email: v6krishn@ucsd.edu; soniamd@ucsd.edu).}}
\date{}
\begin{document}
\sloppy
\thispagestyle{empty}
\pagestyle{empty}
\maketitle
\begin{abstract}
  This work proposes a unifying probabilistic framework for the
  design of robustly asymptotically stable moving-horizon 
  estimators (MHE) for discrete-time nonlinear systems, 
  and a mechanism to incorporate differential privacy
  in moving-horizon estimation. We begin with an
  investigation of the classical notion of strong local observability
  of nonlinear systems and its relationship to optimization-based
  state estimation.  We then present a general moving-horizon
  estimation framework for strongly locally observable systems, as an
  iterative minimization scheme in the space of probability
  measures. This framework allows for the minimization of the
  estimation cost with respect to different metrics.  In particular,
  we consider two variants, which we name $W_2$-MHE and KL-MHE, where
  the minimization scheme uses the 2-Wasserstein distance and the
  KL-divergence, respectively.  The $W_2$-MHE yields a gradient-based
  estimator whereas the KL-MHE yields a particle filter, for which we
  investigate asymptotic stability and robustness properties.
  Stability results for these moving-horizon estimators are derived in
  the probabilistic setting, against the backdrop of the classical
  notion of strong local observability which, to the best of our
  knowledge, differentiates it from other previous works.  We then
  propose a mechanism to encode differential privacy of the 
  measurements used by the estimator via an 
  entropy regularization of
  the MHE objective functional. In particular, we find sufficient
  bounds on the regularization parameter to achieve the desired level
  of differential privacy. Numerical simulations demonstrate the
  performance of these estimators.
\end{abstract}

\section{Introduction}
Moving-horizon estimation (MHE) is an optimization-based state
estimation method that uses the most recent measurements within a
moving-time horizon to recursively update state estimates. % ; and its
% origins can be traced back to limited-memory optimal filters
% introduced in~\cite{AJ:68}.
In principle, its optimization-based formulation enables it to handle
nonlinearities and state constraints much more effectively than other
known methods. This, coupled with the adoption of increasingly
powerful, inexpensive computing platforms has brought new impetus to
the adoption of moving-horizon estimation in various data-driven
applications.  In many cases, data is acquired from particular
individuals or users, which introduces new ethical concerns about data
collection and manipulation, highlighting an increasing need for data
privacy. Such is the case in home monitoring and traffic estimation
(with vehicle GPS data) applications, to name a few. Motivated by
this, here we design and analyze a new class of moving-horizon
estimation filters that can guarantee the differential-privacy of the
data.

The origins of MHE can be traced back to the limited memory optimal
filters introduced in~\cite{AJ:68}. Theoretical investigations on MHE
have broadly been directed at their asymptotic
stability~\cite{CR-JR-DM:03, AA-MB-GB:08, AW-MV-MD:14} and
robustness~\cite{LJ-JBR-WH-AW-MD:16, MM:17, WH:17} properties.  These
properties have primarily been built upon underlying assumptions of
input/output-to-state (IOSS) stability, which is adopted as the notion
of detectability, wherein the norm of the state is bounded given the
sequences of inputs and outputs. However, alternative foundations for
the stability results in other classical notions of observability,
such as strong observability~\cite{HN:82}, have remained unexplored.
The connection between nonlinear observability theory and estimation
problems runs deep, see~\cite{TSL-KPD-CBC:82} and more
recently~\cite{JT-CK:18}, and it is worthwhile to explore this
connection in the context of optimization-based estimation methods
such as moving-horizon estimation.

The problem of state estimation is fundamentally about dealing with
uncertainty, manifested as uncertainty in the initial conditions
and/or in the evolution of the system in the presence of unknown
disturbances. This is appropriately formulated in the space of
probability measures over the state space of the system.  Recent
advances in gradient flows in the space of probability
measures~\cite{LA-NG-GS:08},~\cite{FS:17}, and the corresponding
discrete-time movement-minimizing schemes~\cite{GP:15} present
powerful theoretical tools that can be applied to recursive
optimization-based estimation methods such as moving-horizon
estimation, and can serve as a unifying framework for their design and
analysis.

Another important consideration in the MHE problem is the cost of
computation.  The problem formulation more commonly involves solving
an optimization problem at every time instant, with the state estimate
and disturbances as decision variables in the optimization, where the
dimension of the problem scales with the size of the horizon. This
approach, in general, tends to be computationally intensive, which
poses a hurdle for implementation in real-time.  This has motivated
the search for fast MHE that implement one or more iterations of the
optimization at every time instant. Recently, in~\cite{AA-MG:16},
\cite{AA-MG:17}, the authors develop such a method for noiseless
systems and provide theoretical guarantees on convergence. However,
these works assume the convexity of the cost function, which is
restrictive for general nonlinear systems, and not well connected to
notions of observability. None of these works has considered the
additional question of privacy.

Differential privacy~\cite{CD-AR:14} has emerged over the past decade
as a benchmark in data privacy.  The typical setting assumes
independence between the records in static databases; however, basic
existing mechanisms fail to provide guarantees when correlations exist
between the records in the database.  This is the case when data is
employed by a state estimation process whose output is then released:
there is a dynamic system from which a time series of sensor
measurements is obtained, and the measurement data and the released
estimates are correlated.
%In this case, there exists a map from the state space to the
%(measured) output space, which allows a third party with the knowledge 
%of this map to generate the (true) outputs given the state estimates,
%thereby potentially compromising privacy. 

In~\cite{CD-BN-AM-BR:14, CD-BN-ZZ-AM-BR:17}, the authors generalize
the definition of differential privacy to include general notions of
distance between datasets and design differentially private mechanisms
for Bayesian inference.
%
%\margin{does this extend the KF work by Pappas,
%  etc? is it applicable in our framework?} 
% \marginv{I don't want to directly link it to this work, to give an impression that we are building on it,
% and instead go with the contributions outlined later}
%
In~\cite{CL-SC-PM:16,
  SS-YW-KC:17}, the authors investigate privacy-preserving mechanisms
for the case where correlations exist between database
records.
Privacy-preserving mechanisms for functions and functional data were
investigated in~\cite{RH-AR-LW:13}.  The work~\cite{JLN-GP:14} studies
the problem of differentially-private state estimation, introducing
the formal notion of differential privacy into the framework of Kalman
filter design for dynamic systems. The authors of~\cite{FF-JM-HS:16}
consider the problem of optimal state estimation for linear
discrete-time systems with measurements corrupted by Laplacian
noise.
A finite-dimensional distributed convex optimization is considered
in~\cite{EN-PT-JC:18-tcns}, where differential privacy is achieved by
perturbation of the objective function. We refer the reader
to~\cite{JC-GED-SH-JLN-SM-GJP:16-cdc} for a broad overview of the
systems and control-theoretic perspective on differential privacy.
%
%\subsubsection*{Contributions}
%
%\subsubsection*{MHE contributions}
%

\textit{Contributions:} 
The contributions of this work are two-fold:
establishing the robust asymptotic stability of the
proposed moving-horizon estimator in a 
probabilitstic framework, founded on the notion
of strong local observability; and incorporating
differential privacy in moving-horizon estimation.
We begin with the well-studied
notion of strong local observability of nonlinear, discrete-time
systems and investigate its relationship to the optimization-based
state estimation problem.
To handle uncertain initial conditions and the possible non-uniqueness
of solutions to the estimation problem, we adopt a generalized problem
formulation over the space of probability measures over the state
space.
More precisely, we define  the MHE as a proximal gradient descent in
the space of probability measures, with a non-convex, time-varying
cost function.
This probabilistic setting serves as a unifying framework for
moving-horizon estimation and allows us to develop different classes
of moving-horizon estimators by simply varying the metric used to
define the proximal operator, and to obtain implementable filters by
Monte Carlo methods.
We then consider the Wasserstein metric and the KL-divergence, which
yield the more familiar MHE and a particle filter,
respectively. Following this, we present an analysis of the
convergence and robustness properties of these estimators in the
probabilistic setting, under assumptions of strong local
observability.
%
%\subsubsection*{Privacy contributions}
%
Further, we modify the optimization problem (in the space
of probability measures) by an
entropy regularization to derive conditions that guarantee a desired
level of differential privacy for these filters.

%\subsubsection*{Paper organization}
\textit{Paper organization:} The rest of the paper is organized as follows.  In
Section~\ref{sec:notation_prelims}, we introduce the notation and
mathematical preliminaries used in the paper.  We present the optimization-based
state estimation problem in Section~\ref{sec:optimal_estimation},
where Section~\ref{sec:FIE} deals with the Full Information Estimation
(FIE) problem and the Moving-horizon Estimation (MHE) problem is
introduced in Section~\ref{sec:MHE}.  We present the MHE method based
on proximal gradient descent with the Wasserstein metric in
Section~\ref{sec:W2-MHE}, and with the KL-divergence in
Section~\ref{sec:KL-MHE}.  In Section~\ref{sec:Diff_privacy}, we
address the differential privacy considerations for the moving-horizon
estimators designed.  The results from numerical experiments are
presented in Section~\ref{sec:numerical_expt}, with the conclusions in
Section~\ref{sec:conclusions}.
\section{Notation and preliminaries}
\label{sec:notation_prelims}
In this section, we introduce the notation and
mathematical preliminaries relevant to this paper.

Let~$\| \cdot \| : \real^d \rightarrow \realnonnegative$ denote the
Euclidean norm on~$\real^d$ and~$| \cdot | : \real \rightarrow
\realnonnegative$ the absolute value function.  We denote by $\nabla =
\left( \frac{\partial}{\partial x_1}, \ldots \frac{\partial}{\partial
    x_n} \right)$ the gradient operator in $\real^d$.  For any~$x \in
\mathcal{X} \subset \real^d$, we let~$\mu \in
\mathcal{P}(\mathcal{X})$ be an absolutely continuous probability
measure on~$\mathcal{X} \subset \real^d$.  We denote by~$\rho$ the
corresponding density function, where~$\diff\mu = \rho\dvol$,
with~$\vol$ being the Lebesgue measure.  For~$M \subseteq
\mathcal{X}$, let the distance~$d(x,M)$ of a point~$x \in \mathcal{X}$
to the set~$M$ be given by~$d(x,M) = \inf_{y \in M} \| x - y \|$.  We
denote by~$\left \langle p, q \right \rangle$ the inner product of
functions~$p,q : \mathcal{X} \rightarrow \real$ with respect to the
Lebesgue measure~$\text{vol}$, given by~$\left \langle p,q \right
\rangle = \int_{\mathcal{X}} pq \dvol$.  Let~$F:
\mathcal{P}(\mathcal{X}) \rightarrow \real$ be a smooth real-valued
function on the space of probability measures on~$\mathcal{X} \subset
\real^d$. We denote by~$\frac{\delta F}{\delta \mu} (x)$ the
derivative of~$F$ with respect to~$\mu$, see~\cite{LCE:98}, such that
a perturbation~$\delta \mu$ of the measure results in a
perturbation~$\delta F = \int_{\mathcal{X}} \frac{\delta F}{\delta
  \mu} d(\delta \mu)$. Given a map~$\mathcal{T} : \mathcal{X}
\rightarrow \mathcal{Y}$ and a measure~$\mu \in
\mathcal{P}(\mathcal{X})$, in the space of probability measures
$\mathcal{P}(\mathcal{X})$, we let~$\nu = \mathcal{T}_{\#} \mu$ denote
the pushforward measure of~$\mu$ by~$\mathcal{T}$, where for a
measurable set~$\mathcal{B} \subset \mathcal{T}(\mathcal{X})$, we
have~$\nu(\mathcal{B}) = \mathcal{T}_{\#} \mu (\mathcal{B}) = \mu
(\mathcal{T}^{-1}(\mathcal{B}))$. Moreover, we denote
by~$\mathbb{E}_{\mu}$ the expectation operator w.r.t. the
measure~$\mu$.

We now introduce the notion of~$l$-smoothness 
%\margin{since the $l$ comes from Lipschitz, let's use $L$ instead of $l$} 
%\marginv{I have reverted it to $l$ because I have used it throughout 
% the paper and there is another $L$ which is causing confusion}
that underlies the results on convergence of gradient descent methods.
\begin{definition}\longthmtitle{$l$-smoothness}
  A function $p: \mathcal{X} \rightarrow \real$ is called
  \textit{$l$-smooth} (or \textit{Lipschitz differentiable}) if for
  any~$x, y \in \mathcal{X}$, we have~$|\nabla p(y) - \nabla p(x)|
  \leq l \| y-x \|$.
\end{definition}

The following lemma~\cite{SB-LV:04} can be easily verified for~$l$-smooth functions:
\begin{lemma}\longthmtitle{$l$-smooth functions}
  For an $l$-smooth function $p: \mathcal{X} \rightarrow \real$ and
  any~$x,y \in \mathcal{X}$, we have $| p(y) - p(x) - \langle \nabla
  p(x), y-x \rangle | \leq \frac{l}{2} \| y-x \|^2$.
\label{lemma:cost_func_lipschitz} \oprocend
\end{lemma}

We now define the proximal operator on~$\mathcal{X}$
with respect to a function~$F: \mathcal{X} \rightarrow \real$, as
follows: 
%\margin{Instead of $\mathcal{X}$, which seems to be a finite
%  dimensional space, let's use a different notation for the metric
%  space (as we apply it over $P(\mathcal{X})$). I'm using
%  $\mathcal{S}$ to avoid conflicts further on. BTW, I had to fix this
%  part which was fixed in the conf version. This shows that the latest
%  version of this paper does not contain the changes of the conf
%  paper. That is a waste, I suggest you compare carefully both
%  manuscripts. }\margin{Do we need to square the metric? we don't
%  square it in (5)}
%\marginv{I have changed things to keep it simple, I have defined the proximal operator
% only for the Euclidean space and removed instances of its usage in probability spaces}
\begin{align*}
  \text{prox}_{F} (x) = \arg \min_{\tilde{x} \in \mathcal{X}}
  \frac{1}{2} \| \tilde{x} - x \|^2 + F(\tilde{x}).
\end{align*}

The notion of observability used in this paper is intricately related
to solutions of inverse problems, with an associated notion of
well-posedness that is introduced below:
\begin{definition}\longthmtitle{Well posedness~\cite{AK:11}}
\label{defn:well_posed}
Let~$\mathcal{X}$ and~$\mathcal{Y}$ be normed spaces, and~$P :
\mathcal{X} \rightarrow \mathcal{Y}$ a mapping. The equation~$P(x) =
y$ is called well-posed if:
\begin{enumerate}
\item \textit{Existence}: For every~$y \in \mathcal{Y}$, there is (at
  least one)~$x \in \mathcal{X}$ such that~$P(x) = y$.
\item \textit{Uniqueness}: For every~$y \in \mathcal{Y}$, there is at
  most~$x \in \mathcal{X}$ such that~$P(x) = y$.
\item \textit{Stability}: The solution~$x$ depends continuously
  on~$y$, that is, for any sequence~$\lbrace x_i \rbrace \subset
  \mathcal{X}$ such that~$P(x_i) \rightarrow P(x)$, it follows
  that~$x_i \rightarrow x$.
\end{enumerate}
\end{definition}

We now introduce the notion of lower semicontinuity of set-valued
maps, which underlies some of the results on optimization-based state
estimation in this paper.
\begin{definition}\longthmtitle{Lower semicontinuity of set-valued maps}
  A point-to-set mapping~$H : \mathcal{Z} \subset \real
  \rightrightarrows \real^d$ is lower semicontinuous at a
  point~$\alpha \in \mathcal{Z}$ if for any~$x \in H(\alpha)$ and
  sequences $\lbrace \alpha_i \rbrace \subseteq \mathcal{Z}$, $\lbrace
  x_i \rbrace \subseteq \real^d$ with~$\lbrace \alpha_i\rbrace
  \rightarrow \alpha$, $ \lbrace x_i \rbrace \rightarrow x$ such
  that~$x_i \in H(\alpha_i)$ for
  all~$i$, % there exists a sequence $\lbrace
  % x_i \rbrace \rightarrow x$, where
   it holds that $x \in H(\alpha)$.
  If $H$ is lower semicontinuous at every~$\alpha \in \mathcal{Z}$,
  then~$H$ is said to be lower semicontinuous on~$\mathcal{Z}$.
\end{definition}

%\subsection{Distance between probability measures}
We now define some notions of distance in the space of probability
measures. Let~$\mu_1, \mu_2 \in \mathcal{P} (\mathcal{X})$
be two absolutely continuous probability measures on~$\mathcal{X}$, with~$\rho_1,
\rho_2$ being the corresponding density functions.  Also, let~$\Pi
(\mu_1, \mu_2) \subset \mathcal{P}(\mathcal{X} \times \mathcal{X})$ be the space
of joint probability measures that have~$\mu_1$ and~$\mu_2$ as their
marginals. The~$2$-Wasserstein distance~$W_2(\mu_1, \mu_2)$
between~$\mu_1$ and~$\mu_2$ is given by: 
%\margin{unify notation of $\|\|$ for points in $\mathcal{X}$.}
\begin{align*}
  W_2^2(\mu_1, \mu_2) = \inf_{\pi \in \Pi(\mu_1, \mu_2)} \int_{\mathcal{X}
    \times \mathcal{X}} \| x-y \|^2~d\pi(x,y).
\end{align*}

In what follows, we let~$\frac{\delta W_2^2 (\mu_1, \mu_2)}{\delta
  \mu_1} = \phi_1$, where~$\phi_1$ is the so-called the Kantorovich
potential~\cite{FS:17} associated with the transport from~$\mu_1$
to~$\mu_2$.

The KL-divergence from~$\mu_1$ to~$\mu_2$ is given by:
\begin{align*}
  \subscr{D}{KL}(\mu_1 || \mu_2) = \int_{\mathcal{X}} \log
  \left( \frac{d\mu_1 (x)}{d\mu_2(x)} \right) d\mu_1(x) 
  = \int_{\mathcal{X}} \rho_1(x) \log \left(
    \frac{\rho_1(x)}{\rho_2(x)} \right) \dvol(x).
\end{align*}

The max-divergence between~$\mu_1$ and~$\mu_2$ is defined as:
\begin{align*}
  \subscr{D}{max}(\mu_1, \mu_2) = \sup_{x \in \mathcal{X}} \left| \log\left(
      \frac{\rho_1(x)}{\rho_2(x)} \right) \right|.
\end{align*}

We refer the reader to~\cite{ALG-FES:02} for a detailed overview of
the relations between the various metrics and divergences in
probability spaces.
%
%\subsection{Differential privacy}
%\label{sec:diff_privacy_prelim}

We define an estimator~$\mathcal{E}: \mathcal{Y} \rightarrow
\mathcal{P}(\mathcal{X})$ as a function that accepts as input data $y$
from the metric space~$\mathcal{Y}$ and releases as output
$\mathcal{E}[y]$, a probability measure over the space~$\mathcal{X}$.
\begin{definition}\longthmtitle{Differential privacy}
  Given $\delta$, an estimator~$\mathcal{E}$ is
  \textit{$\epsilon$-differentially private} if for any two
  $\delta$-adjacent measurements~$y_1, y_2 \in \mathcal{Y}$ (that is
  $d_{\mathcal{Y}}(y_1, y_2) \leq \delta$), and any measurable~$A
  \subseteq \mathcal{X}$, we have $\mathcal{E}[y_1](A) \leq
  e^{\epsilon} \mathcal{E}[y_2](A)$.
\label{defn:diff_privacy}
\end{definition}
Note that the condition~$d_{\mathcal{Y}}(y^m_1, y^m_2) \leq \delta$ is
a generalization of the notion of adjacency to arbitrary metric spaces
that we adopt in this paper. We now have the following lemma on the 
connection between the notions
of differential privacy and max-divergence introduced above:
\begin{lemma}\longthmtitle{Differential privacy and max-divergence}
\label{lemma:diff_priv_max_div}
An estimator~$\mathcal{E}$ is $\epsilon$-differentially private
iif~$\subscr{D}{max}(\mathcal{E}[y_1], \mathcal{E}[y_2]) \leq
\epsilon$ for any~$y_1, y_2 \in \mathcal{Y}$ with
$d_{\mathcal{Y}}(y_1, y_2) \leq \delta$.
\end{lemma}
\begin{proof}
  Clearly, if for any~$y_1, y_2 \in \mathcal{Y}$ with
  $d_{\mathcal{Y}}(y_1, y_2) \leq \delta$ we
  have~$\subscr{D}{max}(\mathcal{E}[y_1], \mathcal{E}[y_2]) \leq
  \epsilon$, then:
\begin{align*}
  \epsilon \geq \subscr{D}{max}(\mathcal{E}[y_1], \mathcal{E}[y_2])
  =  \sup_{x \in \mathcal{X}} \left| \log\left( \frac{\rho_1(x)}{\rho_2(x)} \right) \right|
   \geq \left| \log\left( \frac{\rho_1(x)}{\rho_2(x)} \right)
  \right|.
\end{align*}
This implies that for any~$x \in \mathcal{X}$, we have~$\rho_1(x) \leq
e^{\epsilon} \rho_2(x)$.
Now, for any~$A \subseteq \mathcal{X}$, we
have~$\mathcal{E}[y^m_1](A) = \int_{\mathcal{A}} \rho_1(x)\dvol \leq
\int_{\mathcal{A}} e^{\epsilon} \rho_2(x)\dvol = e^{\epsilon}
\int_{\mathcal{A}} \rho_2(x)\dvol = e^{\epsilon}
\mathcal{E}[y^m_2](A)$, which implies that~$\mathcal{E}$
is~$\epsilon$-differentially private. The forward implication can
be easily verified.
\end{proof}
Thus, $\epsilon$-differential privacy essentially imposes an upper
bound on the sensitivity of the estimate generated by~$\mathcal{E}$
(in the sense of the max-divergence~$\subscr{D}{max}$), to the
measurement.
\section{Observability notions}
\label{sec:observability}
In this paper, we consider systems of the form:
\begin{align}
  \begin{aligned}
    \Omega : 
    \begin{cases}
      x_{k+1} = f(x_k, w_k), \\
      y_k = h(x_k) + v_k,         
    \end{cases}
  \end{aligned}
  \label{eq:system_model}
\end{align}
where~$f: \Xsp \times \Wsp \rightarrow \Xsp$ and~$h: \Xsp \rightarrow
\Ysp$, ~$w_k \in \Wsp$ is the process noise,~$v_k \in \Vsp$ is the
measurement noise at time instant~$k$, and~$\Xsp \subset \real^{d_X}$,
$\Ysp \subset \real^{d_Y}$, $\Wsp \subset \real^{d_W}$, and~$\Vsp
\subset \real^{d_V}$. 
%
% \margin{I don't see any reason why we should change notation from
% $\mathcal{X}$ to $\mathbb{X}$ (and the same with the y space). I'm
% replacing the meaning the of the command $\Xsp$ and $Wsp$ etc up
% there; it's an easy change now, keep using the same shorthand
% commands in case of another change. } \marginv{I would like to
% revert to the earlier usage. $\mathcal{X}$ and~$\mathcal{Y}$ have
% been used in the notation section in varied contexts for the purpose
% of definitions, we shouldn't have to be consistent with all of that}
%
%
\begin{assumption}\longthmtitle{Lipschitz continuity}
  The functions~$f$ and~$h$ are Lipschitz continuous, with~$\|
  f(x_1, w_1) - f(x_2, w_2) \| \leq c_f^{(1)} \| x_1 - x_2 \| +
  c_f^{(2)} \| w_1 - w_2 \|$ and $\| h(x_1) - h(x_2) \| \leq c_h \|
  x_1 - x_2 \|$.
	\label{ass:system_lipschitz}
\end{assumption}
\begin{assumption}\longthmtitle{Noise characteristics}
  The noise sequences $\lbrace w_k \rbrace_{k \in \naturals}$
  and~$\lbrace v_k \rbrace_{k \in \naturals}$ are i.i.d samples from
  distributions~$\omega$ and~$\nu$ (with supports in~$\Wsp$
  and~$\Vsp$).  The sets $\Wsp$ and~$\Vsp$ are bounded, with~$|w_k|
  \leq W$ and~$|v_k| \leq V$.  Moreover, we assume
  that~$\mathbb{E}_{\omega}[w_k] = 0$ and~$\mathbb{E}_{\nu}[v_k] = 0$.
\label{ass:noise_char}
\end{assumption}

We also introduce the following autonomous system corresponding
to~\eqref{eq:system_model}:
\begin{align}
  \begin{aligned}
    \Sigma:
    \begin{cases}
      x_{k+1} = f(x_k , 0) = f_0(x_k),\\
      y_k = h(x_k). 
    \end{cases}
  \end{aligned}
  \label{eq:system_model_auto}
\end{align}

With a slight abuse of notation, for any~$x \in \Xsp$, we
let~$\Sigma_{T} (x) = \left( h(x), h \circ f_0(x), \ldots, h \circ
  f_0^T (x) \right)$, the sequence of outputs over a horizon of
length~$T+1$ for the system~\eqref{eq:system_model_auto} from the
state~$x \in \Xsp$.  Similarly, for the
system~\eqref{eq:system_model}, we let~$\Omega(x,\mathbf{w}_{i:j})=
(h(x), h \circ f(x, w_i), \ldots, h \circ f(\ldots f(f(x,
w_i),w_{i+1}),\ldots, w_j)$, 
%\margin{can we use the notation
% $\Omega(x,\mathbf{w}_{i:j})$ to avoid using a double subindex? When
%  the paper is done, it should be easy to replace} 
for some sequence
of process noise samples~$\lbrace w_k \rbrace$, where
$\mathbf{w}_{i:j} = (w_i,\ldots,w_{j})$.

The theoretical results in the moving-horizon estimation literature
have largely been derived in the setting of input/output-to-state
(IOSS) stability, as in~\cite{CR-JR-DM:03, LJ-JBR-WH-AW-MD:16, WH:17}
to name a few, which is a notion of norm-observability,
see~\cite{JH-DL-DA-ES:05}, wherein the norm of the state is bounded
using the sequences of inputs and outputs.
%
%In this paper, we adopt the classical notion of strong local observability,
%to include a class of systems that cannot be globally observed even over an infinite time horizon.
%We present the following example to illustrate this:
%\begin{example}
%Consider the system given by~$x_{k+1} = x_k$ and~$y_k = \sin(x_k)$.
%Let the true initial state be~$x^*_0 = 0$. This implies that for all~$k \in \lbrace 0, 1, \ldots \rbrace$,
%$y_k = 0$. Now, with this output sequence, it is not possible to distinguish
%between sequences of states with initial~$x_0 = n\pi$ for~$n \in \mathbb{Z}$.
%\label{eg:local_obs_example}
%\end{example}
%
%The system in Example~\ref{eg:local_obs_example} is not globally IOSS. 
%We would, however,
%like to expand our class of systems to include such systems that posses only
%local observability properties.
%
However, there are other classical notions of observability based on the 
notion of distinguishability, which generalize the approach taken to linear 
systems. For a detailed treatment, we refer the reader
to~\cite{HN:82} and \cite{FA-DA:02}. 
In this paper, we explore the connection between the classical notion of
strong local observability and moving-horizon estimation.

We now introduce the notion of strong local observability used in this
paper:
\begin{definition}
  \longthmtitle{Strong local observability} The system~$\Sigma$
  defined in~\eqref{eq:system_model_auto} is called \textit{strongly
    locally observable} if there exists a~$T_0 \in \naturals$ such
  that for any given~$x \in \Xsp$ and~$T
  \geq T_0$,
  we have that~${\Sigma_T}^{-1} \circ \Sigma_T(x)$ is a set of
  isolated points. Moreover, for all~$x \in \Xsp$ and $T_1, T_2
  \geq T_0$,
  we have that ${\Sigma_{T_1}}^{-1} \circ \Sigma_{T_1}(x) =
  {\Sigma_{T_2}}^{-1} \circ \Sigma_{T_2}(x)$.  
  We call~$T_0$ the \textit{minimum horizon length}
  of~$\Sigma$.
\end{definition}
%\margin{Is the $T_0$ uniform for all $x$? if so, mention it}
%\marginv{uniformity is implied in the definition}
% 
The above definition is equivalent to the definitions contained
in~\cite{HN:82,FA-DA:02}, which has been restated it in a manner
suitable for the optimization-based estimation framework considered
here. As seen from the above definition, strong observability
is based on a distinguishability notion,
and when it holds globally (i.e., ${\Sigma_T}^{-1} \circ \Sigma_T = id$ for
all $T \geq T_0$) it is equivalent to the notion of uniform observability,
as established in~\cite{SH:10}.

For systems with process noise, of the form~$\Omega$
in~\eqref{eq:system_model}, we introduce the notion of almost sure
strong local observability.
\begin{definition}\longthmtitle{Almost sure strong local
    observability} The
  system~$\Omega$ defined in~\eqref{eq:system_model} is called
  \textit{almost surely strongly locally observable} if there exists
  a~$T^w \in \naturals$ such that, given a process noise sequence
  $\mathbf{w}_{0:T-1} \in \Wsp^T$, for $T \ge T^w$, any $ \yvec_{0:T}
  = \Omega_{\mathbf{w}_{0:T-1}} (x) \in \Ysp^{T+1}$, and~$T \geq T^w$,
  we have that~$\Omega_{\mathbf{w}_{0:T-1}}^{-1}(\yvec_{0:T})$ is a
  set of isolated points almost surely. More precisely, the set of
  noise sequences $\mathbf{w}_{0:T-1}$ for
  which~$\Omega_{\mathbf{w}_{0:T-1}}^{-1}(\yvec_{0:T})$ is not a set
  of isolated points, is of measure zero. Moreover, we call~$T^w$ the
  \textit{minimum horizon length} of~$\Omega$.
\end{definition}
%\margin{is tthe $T^w$ uniform in all $x$? if so, mention it}
%\marginv{implicit in the definition, since it says for any $\yvec_{0:T}
%  = \Omega_{\mathbf{w}_{0:T-1}} (x)$}

We now present a fundamental result that characterizes strong local
observability via a rank condition.
\begin{lemma}
  \longthmtitle{Observability rank condition~\cite{HN:82}} The
  system~$\Sigma$ is locally strongly observable with minimum
  horizon length~$T_0$ if and only if $\text{Rank}(\nabla
  \Sigma_T (x)) = \text{dim}(\Xsp)$ for all~$T \geq T_0$ and~$x \in \Xsp$.  
  The system~$\Omega$ is almost surely locally strongly observable with
  minimum horizon length~$T^w$ if and only if $ \text{Rank}(\nabla
  \Omega_{\mathbf{w}_{0:T-1}} (x)) = \text{dim}(\Xsp)$ almost
  surely for all~$T \geq T^w$ and~$x \in \Xsp$. \oprocend
\label{lemma:strong_obs}
\end{lemma}

We now present an example to illustrate these concepts.
\begin{example}
  Consider a system with the state space~$\Xsp = (0, \infty)$,
  with~$x_{k+1} = f_0(x_k)$ and $y_k = h(x_k)$, such that:
\begin{align*}
  f_0(x) = 
  \begin{cases}
    3x,  \hspace*{0.1in} &\text{for}~ x \in (0, a\pi - \epsilon], \\
     \gamma(x)    \hspace*{0.1in} &\text{for}~ x \in (a\pi - \epsilon, a\pi + \epsilon], \\
    2x + a\pi, \hspace*{0.1in} &\text{for}~ x \in (a\pi + \epsilon, \infty),
  \end{cases}
\end{align*}
for some~$a \in \mathbb{N}$,~$\epsilon$ small and a smooth function~$\gamma$ such that
$\gamma (a\pi - \epsilon) = 3(a\pi - \epsilon)$ and~$\gamma
(a\pi + \epsilon) = 2(a\pi + \epsilon) + a\pi$.
Moreover, let the output~$h(x) = \sin{x}$. We note
that~$\nabla h (x) = \cos{x}$ which implies that~$\nabla h((2m+1)
\pi/2) = 0$ for all~$m \in \mathbb{N}$.  Applying
Lemma~\ref{lemma:strong_obs} for this system, 
%
%\margin{can't we just consider a
%  differentiable system so that we can apply lemma 3 directly?
%  otherwise, and if the definitions should hold uniformly in $x$, what
%  happens at the points of non differentiability?  ---we can't apply
%  our own lemma here. How about the smooth approximation $f(x) =
%  -\frac{1}{\rho} \log [ \text{e}^{-\rho 3x} + \text{e}^{-\rho (2 x +
%    a\pi)}  ]$ for $\rho $ large so that $f_0(\pi/2) $ reduces to
%practically the same value (?)}
%\marginv{fixed the problem a little differently}
%
we can infer that for~$a = 2$, we get
that the minimum horizon length $T_0 = 3$. This is because the system
becomes strongly locally observable at~$x = \pi/2$ only over a horizon
of length~$T_0 = 3$, that is~$\nabla \Sigma_{k} (\pi/2) =
\mathbf{0}_{k+1}$
for~$k \in \lbrace 0, 1, 2 \rbrace$.
This is a case of a one-dimensional system which is strongly locally
observable with a minimum horizon of length~$T_0 = 3$. With larger
values of~$a$, the minimum horizon length is further
increased. \oprocend
\end{example}
We make the following assumption in the rest of the paper:
\begin{assumption}
  \longthmtitle{Strong local observability}
\begin{enumerate}
\item The system~$\Sigma$ in~\eqref{eq:system_model_auto} is strongly
  locally observable with minimum horizon length~$T_0$.
\item The system~$\Omega$ in~\eqref{eq:system_model} is almost surely
  strongly locally observable with minimum horizon length~$T^w$.
  % \item For~$\| W_1 - W_2 \| \leq \delta_W$ and~$\| Y_1 - Y_2 \|
  %   \leq \delta_Y$, $d_H(\Sigma_{W_1}^{-1}(Y_1),
  %   \Sigma_{W_2}^{-1}(Y_2)) \leq \delta$.
\end{enumerate}
\label{ass:strong_loc_obs_regularity}
\end{assumption}
\section{Optimization-based state estimation}
\label{sec:optimal_estimation}
We now begin by addressing the state estimation problem for the
autonomous system~$\Sigma$, and develop a recursive moving-horizon
estimator for it.
% We then show that the estimation error when deployed to the
% system~$\Sigma$, under Assumptions~\ref{ass:system_lipschitz}
% and~\ref{ass:noise_char}, stays bounded over time.
%
\subsection{Full-Information Estimation (FIE)}
\label{sec:FIE}

Let~$\lbrace y_k \rbrace_{k \in \{0\} \cup \naturals}$ be a sequence
of measurements generated by the system~$\Sigma$.  Let~$\lbrace 0,
\ldots, T \rbrace$ be a time horizon such that $T \geq T_0$, the
minimum horizon length of the system~$\Sigma$, and denote~$\yvec_{0:T}
= (y_0, \ldots, y_T)$.  The problem of estimation essentially aims at
characterizing~${\Sigma_T}^{-1}(\yvec_{0:T})$, which is an inverse
problem, and optimal estimation formulates this problem as an
optimization.  Assumptions~\ref{ass:system_lipschitz},
and~\ref{ass:strong_loc_obs_regularity}, on Lipschitz continuity and
strong local observability, respectively, ensure that the inverse
problem is locally well-posed as in Definition~\ref{defn:well_posed}.

To formulate the inverse problem as an optimization, consider a
convex function~$J_T (\yvec_{0:T} , \cdot) : \Ysp^{T+1} \rightarrow
\realnonnegative$ such that~$J_T( \yvec_{0:T} , \xi) = 0$ if and only
if~$\xi = \yvec_{0:T}$. Moreover, we let~$\lim_{T \rightarrow \infty}
J_T( \yvec_{0:T} , \Sigma_T(x)) = \infty$ if~$x \notin {\Sigma_T}^{-1} (\yvec_{0:T})$
for~$T \geq T_0$.
Now, the problem of interest becomes:
\begin{align}
  x_0 \in \arg \min_{x \in \Xsp} J_T(\yvec_{0:T}, \Sigma_T(x)).
  \label{eq:FIE}
\end{align}
In the above,~$\yvec_{0:T}$ is the data in the estimation problem,
which is given.  Since the objective is to solve the original inverse
problem, and we would like to use gradient descent-based methods, we
would like for every local minimizer of~$J_T(\yvec_{0:T},
\Sigma_T(x))$ to belong to the set~${\Sigma_T}^{-1}(\yvec_{0:T})$, or,
in other words, that every local minimizer is also global.  We
therefore make the following additional assumption on the
system~$\Sigma$ and the choice of~$J_T$. For a conciseness of
notation, in the following assumption and lemma, we let~$J_T (\cdot) =
J_T(\yvec_{0:T}, \cdot)$, suppressing the data~$\yvec_{0:T}$ in the
notation where useful, and is understood from context.

\begin{assumption}\longthmtitle{Lower semicontinuity of sublevel sets}
  We assume that, for all~$T \geq T_0$, the convex function~$J_T :
  \Ysp^{T+1} \rightarrow \real$ is such that the set-valued map
  $\mathcal{S}_{\Xsp}(\alpha) = {\Sigma_T}^{-1} \left(
    \mathcal{S}^{J_T}_{\Ysp^{T+1}} (\alpha) \cap \Sigma_T (\Xsp)
  \right)$ is lower semicontinuous,
  where~$\mathcal{S}^{J_T}_{\Ysp^{T+1}} (\alpha) = \lbrace \xi \in
  \Ysp^{T+1} | J_T(\xi) \leq \alpha \rbrace$.
	\label{ass:lower_semicts_set_valued}
\end{assumption}
%
%\margin{this assumption is a bit out of context unless you explain it
%  in the context of theorem 1 of 24. How can we check it? }
%
The above assumption ensures that the function~$J_T \left(
  \yvec_{0:T}, \Sigma_T (\cdot) \right)$ satisfies the condition for
the local minimizers to be global (Theorem~1 from~\cite{IZ-MA:75}).
The following lemma provides a sufficient condition for it to hold.
\begin{lemma}\longthmtitle{Second-order sufficient condition for lower semicontinuity}
  Assumption~\ref{ass:lower_semicts_set_valued} holds if for any~$x
  \in \Xsp$ such that~$\nabla \left( J_T(\yvec_{0:T}, \Sigma_T
    (x))\right) = 0$ we have $J_T(\yvec_{0:T}, \Sigma_T(x)) = 0$, or
  the following condition holds when~$J_T(\yvec_{0:T}, \Sigma_T(x))
  \neq 0$ for any~$v \in \real^{d_X}$,~$v \neq 0$:

{\small
  \begin{align*}
    &\frac{ \left \langle \nabla^2 \Sigma_T [v,v] (x) ,
        \nabla J_T \bigg|_{\Sigma_T(x) } \right \rangle }{\left \| \nabla \Sigma_T [v] \right \|^2} \leq - \subscr{\lambda}{max} \left(\hess
      J_T \bigg|_{\Sigma_T(x)} \right),
  \end{align*} 
}  
\normalsize
where~$\hess J_T$ is the Hessian of~$J_T$. 
\oprocend
\label{lemma:local_global_Hessian}
\end{lemma}
The final inequality in Lemma~\ref{lemma:local_global_Hessian} merely
states that those critical points at which the cost function does not
reach the global minimum value are local maximizers.  
%
%\margin{we
%  haven't really explained the notation in the formula above, that is,
%  the backets and what is that $v$. It is obvious, but we haven't said
%  what Hess is.  I suppose this will imply the composite function to
%  be negative definite? but still what does the notation mean exactly
%  and why do we mutiply by the gradient of the $J_T$?}
%\marginv{I have said what Hess is now. I don't know how to explain the rest except by a full derivation.
% I have left it as it is now, I will add the derivation if it is pointed out in the reviews}
%

We are now ready to present the following theorem that establishes the 
equivalence between the inverse
problem of characterizing the set~${\Sigma_T}^{-1}(\yvec_{0:T})$ and
the optimization~\eqref{eq:FIE}.
\begin{theorem}\longthmtitle{Inverse as minimizer}
  For a convex~$J_T (\yvec_{0:T} , \cdot) : \Ysp^{T+1} \rightarrow
\realnonnegative$ such that~$J_T( \yvec_{0:T} , \xi) = 0$ if and only
  if~$\xi = \yvec_{0:T}$ for any $\yvec_{0:T} \in \Ysp^{T+1}$,
  under Assumptions~\ref{ass:strong_loc_obs_regularity}
  and~\ref{ass:lower_semicts_set_valued}, and any~$T \geq T_0$, it
  holds that $z \in {\Sigma_T}^{-1}(\yvec_{0:T})$ if and only if~$z $
  is a minimizer of $J_T(\yvec_{0:T}, \Sigma_T(\cdot))$.
\label{thm:est_optimization}
\end{theorem}
\begin{proof}
  If~$z \in {\Sigma_T}^{-1}(\yvec_{0,T})$, we have that $h \circ f^k_0
  (z) = y_k$ for all~$k \in \lbrace 0, \ldots, T \rbrace$. It now
  follows that~$J_T(\yvec_{0:T}, \Sigma_T(z)) = 0$. 
  Since,~$J_T(\yvec_{0:T}, \Sigma_T(z)) \geq 0$ by definition, we
  infer that~$z $ is a global minimizer of $J_T(\yvec_{0:T},
  \Sigma_T(\cdot))$.

  Suppose that~$z $ is a local minimizer of $J_T(\yvec_{0:T}, \Sigma_T(\cdot))$. By
  Assumption~\ref{ass:lower_semicts_set_valued} and Theorem~1
  in~\cite{IZ-MA:75}, we get that the local minima of~$J_T(\yvec_{0:T}, \Sigma_T(\cdot))$ 
  are also global, which implies that~$J_T(\yvec_{0:T}, \Sigma_T(z)) = 0$, and therefore
  $\Sigma_T (z) = \yvec_{0:T}$.
\end{proof}
Theorem~\ref{thm:est_optimization} suggests that the state estimates
for the system~$\Sigma$ can be obtained by minimizing~$J_T(\yvec_{0:T}, \Sigma_T(\cdot))$ 
over a horizon of length~$T \geq T_0$.  This is also
called the full information estimation (FIE) problem in the optimal
state estimation literature~\cite{CR-JR-DM:03,WH:17},
as it works with the entire sequence of output measurements over the
horizon~$\lbrace 0, \ldots, T \rbrace$.
%
%\begin{assumption}\longthmtitle{}
%%
%  \margin{name for the assumption? it's a sort of 'dynamic separation'
%    of trajectories starting in $\left( J_T \circ \Sigma_T
%    \right) (x) = 0$. How about 'Dynamic-separation of critical
%    trajectories and stability of the input set' (?), where the input
%    set is ${\Sigma_T}^{-1}(\yvec)$, which we can name when we
%  introduce it. I don't know if 'input set' is the best name, it
%  interprets the map as taking 'inputs' (or should rather be 'states')
%  and transforming them into 'outputs'. ('State-to-output set'?) After
%  it, we can explain briefly in words what it means} 
%%
%Let~$\mathcal{S}= \left \lbrace x \in \Xsp~|~ \nabla \left( J_T \circ
%    \Sigma_T \right) (x) = 0 \right \rbrace \setminus
%{\Sigma_T}^{-1}(\yvec)$ and~$\mathcal{C}$ be an open, bounded
%subset of the basin of stability of~${\Sigma_T}^{-1}(\yvec)$. We
%assume that $d_H(f^k_0(\mathcal{S}), f^k_0(\mathcal{C})) \geq \beta
%d_H(\mathcal{S}, \mathcal{C})$ for all~$k \in \mathbb{N}$ and
%some~$\beta > 0$.
%\label{ass:basin_stability}
%\end{assumption}
%%
%\margin{Is $\mathcal{C}$ *any* bounded subset of this basin? is it a
%  subset or a superset? but it does contain $\setminus
%  {\Sigma_T}^{-1}(\yvec)$, right?}
%%
%\margin{At some point, recall we need to comment on the verification
%  of these assumptions and/or what does it mean if they don't
%  hold. Can we use Lyapunov functions and Lyapunov-like results for
%  this?}
%

Now, from Assumption~\ref{ass:strong_loc_obs_regularity} and
Theorem~\ref{thm:est_optimization}, we have
that~${\Sigma_T}^{-1}(\yvec_{0:T})$ is a set of isolated points which
are %local
minimizers of~$J_T(\yvec_{0:T}, \Sigma_T(\cdot))$. It then
follows that~${\Sigma_T}^{-1}(\yvec_{0:T})$ is the set of stable fixed
points of the negative gradient vector field of~$J_T(\yvec_{0:T},
\Sigma_T(\cdot))$.  We let~$\mathcal{C}_0$ be the basin of attraction
of this set. Moreover, we note
that~$f^k({\Sigma_T}^{-1}(\yvec_{0:T}))$ is the set of stable fixed
points of the negative gradient vector field of~$J_T
\left(\yvec_{k:k+T}, f^k \circ \Sigma_T (\cdot) \right)$, and we
let~$\mathcal{C}_k$ be the basin of attraction
of~${\Sigma_T}^{-1}(\yvec_{k:k+T})$. 
We have used above the fact
that~${\Sigma_T}^{-1}(\yvec_{k:k+T}) =f_0^k({\Sigma_T}^{-1}(\yvec_{0:T}))$,
which follows from the definition of strong local observability. 

We now lift the FIE problem~\eqref{eq:FIE} to the space of probability
measures over~$\Xsp$, as a minimization in expectation of the
estimation objective function:
\begin{align}
  \mu_0 \in \arg \min_{\mu \in \mathcal{P}(\Xsp)} \mathbb{E}_{\mu}
  \left[ J_T(\yvec_{0:T}, \Sigma_T(\cdot)) \right].
	\label{eq:FIE_prob_meas}
\end{align}
The above formulation allows us to capture information about the
(probably many) optimal estimates through a probability
measure~$\mu_0$, 
%which will help encode the differential privacy
%constraint in the ensuing sections. 
%
%\margin{we can say ``the above
%  formulation allows us to capture the information about the (probably
%  many) optimal estimates through a probability measure~$\mu_0$, and
%  also help encode other constraints on the resulting pdfs such as
%  differential privacy.'' and we can say we leave that study for a
%  forthcoming publication}
%
and help encode distributional constraints,
which will be considered in a forthcoming
publication. 

In the following, we develop recursive moving-horizon estimators that
generate sequences~$\lbrace \mu_k \rbrace_{k \in \naturals}$ of
probability measures in~$\mathcal{P}(\Xsp)$ as estimates. We then
obtain practically implementable estimators using Monte Carlo methods
to sample from the measures~$\mu_k$.
\subsection{Moving-Horizon Estimation (MHE)}
\label{sec:MHE}
In the previous section, we
presented a formulation of the full information estimation (FIE)
problem for the autonomous system~$\Sigma$, which uses the entire
measurement sequence over a horizon
%~$\lbrace 0, \ldots, T \rbrace$, 
of length~$T \geq T_0$.  However, the minimum horizon length~$T_0$ may
be large, which would make the estimation computationally
intensive. Moreover, we would like to progressively assimilate the 
incoming measurements online.
We therefore adopt a moving-horizon estimation method which, at any
time instant~$k+N$, uses the output measurements from the
horizon~$\lbrace k+1, \ldots, k+N \rbrace$ (of length~$N < T_0$), and
the state estimate at the time instant~$k-1$, to obtain the state
estimate at instant~$k$, recursively.

We let~$G^N_k(z) = J_{N-1} \left( \mathbf{y}_{k+1:k+N}, \Sigma_N (z)
\right)$ be the objective function over the horizon~$\lbrace k+1,
\ldots, k+N \rbrace$, at the time instant~$k+N$,
where~$\mathbf{y}_{k+1:k+N} = (y_{k+1}, \ldots, y_{k+N})$.
%
% \margin{should this be $G^N_k(z) = J_{N-1} \circ
% \Sigma^0_{k+1,N-1}(z)$, as the latter $\Sigma_{k+1,N-1}^0$ is
% understood as happening from $k+1$ to $k+N$? (from the early
% notation we introduced)}
%
% \margin{If we change the notation of $\Sigma^0$, then we can change
% the notation $G^N_k(z)$ to $G_k^N(z)$}
%
\begin{assumption}\longthmtitle{Moving-horizon cost}
  We make the following assumptions on the cost function~$G^N_k$:
  \begin{enumerate}
    % \item $L$-smoothness: $\| \nabla G^N_k(y) - \nabla G^N_k(x) \| \leq l \| y-x \|$.
  \item the cost $G^N_k$ is~$l$-smooth,
  \item it holds that $| G^N_{k+1}(f_0(z)) - G^N_k(z) | \leq L \| \nabla G^N_k(z) \|^2$,
  \item the previous constants are such that~$lL \leq \frac{1}{2}$,
  \item for any two~$\delta$-adjacent measurements~$\yvec,
    \tilde{\yvec} \in \Ysp^{T+1}$, such that~$\| \yvec - \tilde{\yvec}
    \| \leq \delta$ 
%    \margin{sometimes we use norm, sometimes we use
%      the absolute value symbol ($|\cdot |$) can we just use one? or
%      what is the criterion to pick one? } 
      and with corresponding costs~$G^N_k$
    and~${\widetilde{G}}^N_k$, for~$k \in \lbrace 0, \ldots, T \rbrace$
    and~$N \leq T-k$, we have~$\| \nabla (G^N_k -
    {\widetilde{G}}^N_k)(x) \| \leq l \delta$ for all~$x \in \Xsp$.
  \end{enumerate}
  \label{ass:cost_func_lipschitz}
\end{assumption}

We now formulate the general moving-horizon estimation method as
follows: 
%\margin{do we need to square the divergence? this was a
%  question in the conf paper. I believe the latest version of the conf
%  does not correspond to this version, which is a waste. Please
%  compare both versions to condense to the best manuscript version.}
%  \marginv{We don't need to square the divergence}
\begin{align}
	\begin{aligned}
    &\mu_k \in \arg \min_{\mu \in \mathcal{P}(\Xsp)} 
    D(\mu, f_{0\#} \mu_{k-1})
    + \eta \mathbb{E}_{\mu} \left[ G^N_k \right], \\
    &\text{given}~~\mu_0 \in \mathcal{P}(\Xsp),
  \end{aligned}
  \label{eq:MHE_unifying_formulation}
\end{align}
where~$D: \mathcal{P}(\Xsp) \times \mathcal{P}(\Xsp) \rightarrow
\realnonnegative$ is a placeholder for a metric, divergence or
transport cost on~$\mathcal{P}(\Xsp)$.  We obtain implementable
observers  from the above
formulation by sampling from the measures, by Monte Carlo methods.  As
discussed in the ensuing sections, using the~$2$-Wasserstein
distance~$W_2$ yields the more familiar MHE formulation, whereas with
the~KL-divergence we obtain a moving-horizon particle filter. Hence,
this formulation is proposed as a unifying probabilistic framework
for moving-horizon estimation, where different estimators are
generated by different choices of~$D$.

We now introduce the following asymptotic stability notion for
estimators that will be used in investigating the properties of the
estimators we design.
\begin{definition}\longthmtitle{Asymptotic stability of state estimator}
  We call an estimator of the form~\eqref{eq:MHE_unifying_formulation} 
  an \textit{asymptotically stable observer} for
  the system~$\Sigma$ if the sequence of estimates~$\lbrace \mu_k
  \rbrace_{k \in \mathbb{N}}$ is such that $\lim_{k \rightarrow
    \infty} \mu_k({\Sigma_T}^{-1}(\yvec_{k:k+T})) = 1$ for~$T
  \geq T_0$.
\end{definition}
\section{A $W_2$-Moving-Horizon Estimator}
\label{sec:W2-MHE}
In this section, we derive a moving-horizon estimator, which we refer
to as the $W_2$-MHE, to generate a sequence of probability
distributions~$\lbrace \mu_k \rbrace_{k \in \naturals}$. This is based
on the one-step minimization scheme of~\cite{FS:17}
in~$\mathcal{P}(\Xsp)$ w.r.t.~the Wasserstein metric~$W_2$, which 
we extend to the moving-horizon setting. 
%and show how to devise an associated
%Markov-Chain Monte Carlo scheme~\cite{CR-GC:13}, which we analyze. 
For every $k > 0$, consider:
\begin{align}
  \begin{aligned}
    &\mu_k \in \arg \min_{\mu \in \mathcal{P}(\Xsp)} \frac{1}{2}
    W_2^2(\mu, f_{0\#} \mu_{k-1})
    + \eta \mathbb{E}_{\mu} \left[ G^N_k \right], \\
    &\text{given}~~\mu_0 \in \mathcal{P}(\Xsp).
  \end{aligned}
  \label{eq:MHE_wasserstein}
\end{align}
We let~$\mathcal{K}_k$ be the support of~$\mu_k$, with~$\mathcal{K}_0
\subseteq \mathcal{C}_0$, where~$\mathcal{C}_0$ is as defined earlier
in Section~\ref{sec:FIE}.

%We represent the above in a compact form using the proximal operator
%on~$\mathcal{P}(\Xsp)$ associated with $\mathbb{E}_\mu \left[ G^N_k
%\right] $ and w.r.t.~the Wasserstein metric~$W_2$, which we denote as
%follows:
%\begin{align}
%  \mu_k \in \text{prox}^{W_2}_{\eta G^N_k } \left(f_{0\#}  \mu_{k-1}
%  \right),  \quad k > 0.
%  \label{eq:MHE_prox_map}
%\end{align}
% 
%The objective functional in~\eqref{eq:MHE_wasserstein} is not
%necessarily convex, which implies that the minimizer is not
%necessarily unique.
%
\subsection{Sample update scheme for~$W_2$-MHE}
We now derive a sample update scheme for~$W_2$-MHE, which
also yields an implementable filter for the~$W_2$-MHE formulation.

We note that any local minimizer~$\mu_k$ of~\eqref{eq:MHE_wasserstein} is
a critical point of the objective functional and therefore, it
satisfies:
\begin{align*}
  c = \frac{\delta}{\delta \mu} \left[ \left( \frac{1}{2} W_2^2(\mu,
      f_{0\#} \mu_{k-1}) + \eta \mathbb{E}_{\mu} \left[ G^N_k \right]
    \right) \right] \bigg|_{\mu = \mu_k} = \phi_k + \eta G^N_k,
\end{align*}
where~$\phi_k$ is the Kantorovich potential~\cite{FS:17} associated with the transport from~$\mu_k$
to~$f_{0\#} \mu_{k-1}$, and~$c$ is a constant (from the
constraint~$\int_{\Xsp} d\mu(x) = 1$, for $\mu \in \mathcal{P}(\Xsp)$,
due to which the first variation is defined up to an additive
constant).  From the above equation, we now obtain:
\begin{align*}
  \nabla \phi_k(x) + \eta \nabla G^N_k(x) = 0.
\end{align*}
The gradient of the Kantorovich potential~$\phi_k$ defines the deterministic
optimal transport map~$T_k$ (note that this notation is not to be confused 
with that of the time horizon~$T$) 
w.r.t. the~$W_2$-distance
from~$\mu_k$ to~$f_{0\#}\mu_{k-1}$, which determines~$\nabla \phi_k
(x) = x - T_k^{-1}(x)$ (where~$\mu_k = {T_k}_{\#} f_{0\#}\mu_{k-1}$).
%
%\margin{should we explictly say that $T_k$ is the solution to the
%  deterministic optimal transport problem from $\mu_k$ to
%  $f_{0,\sharp}\mu_{k-1}$?Note: $T_k$ may be confused with the $T$ of
%  the time horizon; make a note not to confuse it with it. } 
%
  We therefore get:
\begin{align}
  x = T_k^{-1}(x) - \eta \nabla G^N_k(x).
  \label{eq:transport_ray_MHE}
\end{align}
The above equation allows us to design an implementable
filter for the $W_2$-MHE~\eqref{eq:MHE_wasserstein}.  
We let~$z_k \sim \mu_k$, that is, $z_k \in \mathcal{K}_k$ is sampled from
the distribution~$\mu_k$. From~\eqref{eq:transport_ray_MHE}, 
it holds that~$z_k = T_k^{-1}(z_k) - \eta \nabla G^N_k(z_k)$. Since~$(T_k^{-1})_{\#}
\mu_k = f_{0 \#}\mu_{k-1}$, we let~$T_k^{-1}(z_k) = f_0(z_{k-1})$,
%
%\margin{vishaal, pls explain a bit more what this means. By 'allow'
%  you mean 'we can take $z_k = T_k \circ f_0(z_{k-1})$ and it will
%  satisfy (7)? (is $T_k$ set valued because of the argmin?  should (7)
%  have a $\in$ instead of an $=$?) at the end it should be that $z_k$
%  satisfying the equation (8), where $z_{k-1}$ is sampled according to
%  $\mu_{k-1}$ should correspond to a random sample according to
%  $\mu_k$, correct?  } 
%  
a sample of the distribution~$f_{0\#}\mu_{k-1}$,
and we obtain the following recursive estimator:
\begin{align}
  z_{k} = f_0(z_{k-1}) - \eta \nabla G^N_k(z_{k}), \quad k>0.
  \label{eq:MHE_rec_est}
\end{align}
We now note that the estimate~$z_k$ in~\eqref{eq:MHE_rec_est} corresponds to a
critical point of the following minimizing movement scheme:
\begin{align}
  \begin{aligned}
    &z_{k} \in \arg \min_{z} \frac{1}{2} \left \| z - f_0(z_{k-1})
    \right\|^2
    + \eta G^N_k(z), \quad k > 0, \\
    &z_0 \sim \mu_0 \in \mathcal{P}(\Xsp).
  \end{aligned}
  \label{eq:MHE_opt}
\end{align}
\begin{lemma}\longthmtitle{Strong convexity}
  For~$\eta < l^{-1}$, the objective function in~\eqref{eq:MHE_opt} is
  strongly convex, and therefore~$\text{prox}_{\eta G^N_k}(f_0(x))$ is a
  singleton for any~$x \in \Xsp$.
\end{lemma}
\begin{proof}
  Let~$\Theta(z) = \frac{1}{2} \left\| z - f_0(\tilde{z}) \right\|^2 +
  \eta G^N_k(z)$.  We have that~$\nabla \Theta(z_1) - \nabla
  \Theta(z_2) = z_1 - z_2 + \eta \left( \nabla G^N_k(z_1) -\nabla
    G^N_k(z_2) \right)$.  It now follows that~$\left \langle \nabla
    \Theta(z_1) - \nabla \Theta(z_2) , z_1 - z_2 \right \rangle = \|
  z_1 - z_2 \|^2 + \eta \left \langle \nabla G^N_k(z_1) -\nabla
    G^N_k(z_2) , z_1 - z_2 \right \rangle$.  From
  Assumption~\ref{ass:cost_func_lipschitz}-(1), on the moving-horizon
  cost, we now get that~$\left \langle \nabla \Theta(z_1) - \nabla
    \Theta(z_2) , z_1 - z_2 \right \rangle \geq (1 - \eta l) \| z_1 -
  z_2 \|^2$, and since~$\eta l < 1$, we infer that~$\Theta$ is
  strongly convex, and therefore has a unique
  minimizer. Thus,~$\text{prox}_{\eta G^N_k} (f_0(\tilde{z})) = \arg
  \min_z \Theta(z)$ is a singleton.
\end{proof}

We note that the minimization~\eqref{eq:MHE_opt} defines a proximal
mapping w.r.t. the Euclidean metric, which we represent in a compact
form using the proximal operator as:
\begin{align}
  \begin{aligned}
    &z_k = \text{prox}_{\eta G^N_k} (f_0(z_{k-1})), \quad k > 0, \\
    &z_0 \sim \mu_0 \in \mathcal{P}(\Xsp),
  \end{aligned}
  \label{eq:MHE_rec_prox}
\end{align}
where~$\text{supp}(\mu_0) = \mathcal{K}_0 \subseteq \mathcal{C}_0$.
\subsection{Asymptotic stability of~$W_2$-MHE}
We present the asymptotic stability result for~$W_2$-MHE in this
section, before which we introduce the following assumption on
positive invariance of the discrete-time dynamics defined by the
map~$\text{prox}_{\eta G^N_k} \circ f$.
\begin{assumption}\longthmtitle{Positive invariance}
  We assume that there exists~$\alpha > (1 - \sqrt{1-2lL}){l}^{-1}$
  such that for all~$\eta \in (0, \alpha)$, we have~$\text{prox}_{\eta
    G^N_k} ( f (\mathcal{C}_{k-1}) ) \subseteq \mathcal{C}_k$.
\label{ass:basin_stability}
\end{assumption}

The above assumption ensures that under the discrete-time dynamics defined by
the map~$\text{prox}_{\eta G^N_k} \circ f$, any sequence starting in the basin of 
attraction~$\mathcal{C}_0$ of~${\Sigma_T}^{-1}(\yvec_{0:T})$ remains within the 
basins of attraction~$\mathcal{C}_k$ of~${\Sigma_T}^{-1}(\yvec_{k:k+T})$
at the subsequent instants of time~$k \in \naturals$.

We are now ready to present the asymptotic stability result for~$W_2$-MHE:
\begin{theorem}\longthmtitle{Asymptotic stability of~$W_2$-MHE}
\label{thm:MHE_asymp_stability}
The estimator~\eqref{eq:MHE_wasserstein}, under
Assumptions~\ref{ass:strong_loc_obs_regularity} to~\ref{ass:basin_stability},
with a constant step size~$\eta \in \displaystyle{ \left( \frac{1 - \sqrt{1-2lL}}{l}, \min\left\lbrace \alpha, \frac{1}{l} \right \rbrace \right)}$,
is an asymptotically stable observer for the system~$\Sigma$.
\end{theorem}
\begin{proof}
  By Assumption~\ref{ass:cost_func_lipschitz}-$(1)$, on the
  moving-horizon cost, and~Lemma~\ref{lemma:cost_func_lipschitz}, we
  have:
\begin{align*}
  | G^N_k(f_0(z_{k-1})) - G^N_k(z_k) -
  \langle \nabla G^N_k(z_k), f_0(z_{k-1}) - z_k \rangle | 
  \leq \frac{l}{2} \| f_0(z_{k-1}) - z_k \|^2.
\end{align*}
Substituting from~\eqref{eq:MHE_rec_est} into the above, we get:
\begin{align*}
  | G^N_k(f_0(z_{k-1})) - G^N_k(z_k) 
  - \eta \| \nabla G^N_k(z_k) \|^2 | 
  \leq \eta^2 \frac{l}{2} \| \nabla G^N_k(z_k) \|^2.
\end{align*}
It now follows that:
\begin{align*}
  G^N_k(z_k) \leq G^N_k(f_0(z_{k-1})) - \eta \left( 1 - \frac{l}{2}
    \eta \right) \| \nabla G^N_k(z_k) \|^2.
\end{align*}
From Assumption~\ref{ass:cost_func_lipschitz}-$(2)$, on the moving-horizon
cost, we have:
\begin{align*}
  G^N_k(z_k) \leq ~G^N_{k-1}(z_{k-1}) + L \| \nabla G^N_{k-1}(z_{k-1}) \|^2 
   - \eta \left( 1 - \frac{l}{2} \eta \right) \| \nabla G^N_k(z_k)
  \|^2.
\end{align*}
Summing the above inequality from~$k=1$ to~$K$, we get:
\begin{align*}
  \eta \left( 1 - \frac{l}{2} \eta \right) \sum_{k=1}^K
  \| \nabla G^N_k(z_k) \|^2 - L \sum_{k=1}^K \| \nabla G^N_{k-1}(z_{k-1}) \|^2 
  \leq G^N_0(z_0) - G^N_K(z_K).
\end{align*}
From here, we obtain:
\begin{align*}
  \left[ \eta \left( 1 - \frac{l}{2} \eta \right) - 
    L \right] \sum_{k=1}^K \| \nabla G^N_k(z_k) \|^2 
  &\leq G^N_0(z_0) - G^N_K(z_K) + L \| \nabla G^N_0(z_0) \|^2 \\
  &\leq G^N_0(z_0) + L \| \nabla G^N_0(z_0) \|^2.
\end{align*}
Since~$\eta \in \displaystyle{\left( \frac{1 - \sqrt{1-2lL}}{l},
    \frac{1}{l} \right)}$,
% ~$\eta \in \left( \frac{1 - \sqrt{1-2lL}}{l}, \frac{1 + \sqrt{1-
%   2lL}}{l} \right)$,
we have that $\eta \left( 1 - \frac{l}{2} \eta \right) - L > 0$ and
therefore, taking limits in the previous inequality, we deduce that
the series is summable. The latter implies that~$\lim_{k \rightarrow
  \infty} \nabla G^N_k(z_k) = 0$, and from~\eqref{eq:MHE_rec_est}, we
have that $\lim_{k \rightarrow \infty} \|z_k - f(z_{k-1})\| = 0$.
  
It now follows, by definition, from the above that $\lim_{k \rightarrow \infty} \nabla
G_k^{T+1} (z_k) = \lim_{k \rightarrow \infty} \nabla \left( J_{T}
  \left( \mathbf{y}_{k:k+T} , \Sigma_T (z_k) \right) \right) = 0$,
over a horizon of length~$T+1$ (with~$T \geq T_0$).
%
%\margin{I forget
%  why this follows for $T \ge T_0 +1$, it must be by the limit in
%  $z_k$ above and the definition of the $J_T$. Can you briefly state
%  it here? it can be confusing as we had a $N < T_0$}  
%  
We now have that
the initial condition~$z_0 \in \mathcal{K}_0 \subseteq \mathcal{C}_0$ and
Assumption~\ref{ass:basin_stability} ensure that~$z_k \in
\mathcal{C}_k$, the basin of attraction of~$f^k \left(
  {\Sigma_T}^{-1}(\yvec_{0:T}) \right)$ and from the fact
that~$\lim_{k \rightarrow \infty} \nabla \left( J_{T} \left(
    \mathbf{y}_{k:k+T} , \Sigma_T (z_k) \right) \right) = 0$, we infer
that~$\lbrace z_k \rbrace$ converges to the local minima of~$J_{T}
\left( \mathbf{y}_{k:k+T} , \Sigma_T (\cdot) \right) $.  By
Theorem~\ref{thm:est_optimization}, it now follows that~$\lbrace z_k
\rbrace$ converges to the set~$\Sigma_T^{\quad -1}(\mathbf{y}_{k:k+T})$.
Therefore~$\lim_{k \rightarrow \infty} d (z_k ,
{\Sigma_T}^{-1}(\yvec_{k:k+T})) = 0$.

Moreover, since~$\lim_{k \rightarrow \infty} d (z_k ,
{\Sigma_T}^{-1}(\yvec_{k:k+T})) = 0$ for all~$z_0 \in \mathcal{K}_0$,
it follows that~$\lim_{k \rightarrow \infty} \mathcal{K}_k =
{\Sigma_T}^{-1}(\yvec_{k:k+T})$. We know that~$\supp(\mu_k) =
\mathcal{K}_k$, and therefore we get that~$\lim_{k \rightarrow \infty}
\mu_k \left( {\Sigma_T}^{-1}(\yvec_{k:k+T}) \right) = 1$. 
%
%  \margin{why does the last limit
%  hold (with $T+1$ instead of $N$)? is it because of a similar
%  argument to that with $N$? Just add a small phrase here on this} 
%
%\margin{For me it was
%  helpful to do this: if $\bar{\yvec}_{k:T+k}$ is generated by the
%  data and contained in $\bar{\yvec}_{0,T}$. This should be generated
%  by one of the elements in $\Sigma_{0,T}^-1 (\yvec_{0,T})$, then, we
%  can view $\bar{\yvec}_{k:T+k}$ as the output data
%  $h(f_0^{k+l}(x_k))$ for some $x_k\in f_0^k(\Sigma_{0,T}^-1
%  (\yvec_{0,T}))$ and $l \le T$. Take this $\bar{\yvec}_{k:k+T}$ and
%  use it to define $J_T \circ \Sigma_{k,T}^0$. The set of minimizers
%  of $J_T\circ \Sigma_{k,T}^0$ should coincide with
%  $f_0^k(\Sigma_{0,T}^{-1}(\yvec)$. Because
%  $\lim_{k\rightarrow\infty}\nabla G_{k,T+1}(z_k) = 0$, and assumption
%  4, then the convergent series should approach the aforementioned
%  set. If this argument is correct, can you add some of the
%  information into the proof?}  
%  
\end{proof}

\subsection{Robustness of $W_2$-MHE}
We now characterize the performance of the
estimator~\eqref{eq:MHE_wasserstein} on the system~$\Omega$
in~\eqref{eq:system_model}.  Since the true process and measurement
noise sequences remain unknown, we are interested in the robustness
properties of the estimator~\eqref{eq:MHE_wasserstein_ref}, in the
form of an upper bound by the norms of the disturbance sequences on
the estimation error.

We begin by constructing a reference estimator that recursively
generates the estimate sequence, given the true disturbance
sequences~$\lbrace w_k \rbrace_{k \in \naturals}$ and~$\lbrace v_k
\rbrace_{k \in \naturals}$, as follows:
\begin{align}
  \begin{aligned}
    &\bar{\mu}_k \in \arg \min_{\mu \in \mathcal{P}(\Xsp)} \frac{1}{2}
    W_2^2(\mu, f_{0\#} \bar{\mu}_{k-1})
    + \eta \mathbb{E}_{\mu} \left[ \bar{G}^N_k \right], \\
    &\text{given}~~\bar{\mu}_0 \in \mathcal{P}(\Xsp).
  \end{aligned}
  \label{eq:MHE_wasserstein_ref}
\end{align}
where, we employ for conciseness $\mathbf{w}\equiv\mathbf{w}_{k:k+N-1}
= (w_k,\dots,w_{k+N-1})$ and $\mathbf{v} \equiv \mathbf{v}_{k+1:k+N} =
(v_{k+1}, \dots, v_{k+N})$, so that $\bar{G}^N_k(z) \equiv
\bar{G}^N_k(z,\mathbf{w},\mathbf{v}) = J_{N-1} \left(
  \mathbf{y}_{k+1:k+N} , \Omega_{\mathbf{w}_{k:k+N-1}} (z) +
  \mathbf{v}_{k+1:k+N} \right)$.  Note that~$G^N_k = \bar{G}^N_k
\big|_{\mathbf{w}=0, \mathbf{v}=0}$.  We let~$\bar{\mathcal{K}}_k$ be
the support of~$\bar{\mu}_k$, with~$\bar{\mathcal{K}}_0 \subseteq
\bar{\mathcal{C}}_0$, where the definition of~$\bar{\mathcal{C}}_k$ is
similar to that of~$\mathcal{C}_k$ but taking the noise~$\lbrace w_k \rbrace$
and $\lbrace v_k \rbrace$ into account.
\begin{assumption}\longthmtitle{l-Smoothness w.r.t. disturbances}
  We assume that~$\| \nabla G^N_k (z) - \nabla \bar{G}^N_k(z) \|
  \leq l_w \| (\mathbf{w}_{k:k+N-1},\mathbf{v}_{k+1:k+N}) \|$ for all~$z \in \Xsp$.
\label{ass:bar_G}
\end{assumption}

Following the proof of Theorem~\ref{thm:MHE_asymp_stability}, under
the same set of underlying assumptions, we infer that the reference
estimator~\eqref{eq:MHE_wasserstein_ref} is almost surely an
asymptotically stable observer for the system~$\Omega$, given a particular
realization of the disturbances~$\lbrace w_k \rbrace_{k \in \naturals}$
and~$\lbrace v_k \rbrace_{k \in \naturals}$.

We now present the following theorem on the robustness of the
estimator~\eqref{eq:MHE_wasserstein}, characterized by a bound on the
error in the estimates generated by~\eqref{eq:MHE_wasserstein} with
respect to the estimates generated by the reference
estimator~\eqref{eq:MHE_wasserstein_ref}:
\begin{theorem}\longthmtitle{Robustness of~$W_2$-MHE}
  Under
  Assumptions~\ref{ass:system_lipschitz},~\ref{ass:strong_loc_obs_regularity},~\ref{ass:cost_func_lipschitz},
  and~\ref{ass:bar_G}, given the estimate sequences~$\lbrace \mu_k
  \rbrace_{k \in \naturals}$ generated by~\eqref{eq:MHE_wasserstein}
  and $\lbrace \bar{\mu}_k \rbrace_{k \in \naturals}$ generated by the
  reference estimator~\eqref{eq:MHE_wasserstein_ref}, with~$\mu_0 =
  \bar{\mu}_0$, we have~$W_2(\mu_k, \bar{\mu}_k) \leq
  \frac{c_f^{(2)}}{c_f^{(1)}} W C_k + \frac{\eta l_w \sqrt{N}
  }{c_f^{(1)}} (W + V) C_k$, for all~$k \in \mathbb{N}$, where~$C_k =
  \sum_{\ell=1}^k (\frac{c_f^{(1)}}{1 - \eta l} )^\ell $.
  %Moreover, when~$z_0 \neq \bar{z}_0$, the inequality
  %holds asymptotically, as~$k \rightarrow 0$. 
\end{theorem}
\begin{proof}
The estimator~\eqref{eq:MHE_wasserstein_ref}
yields the following reference recursive scheme:
\begin{align}
  \bar{z}_k = f(\bar{z}_{k-1}, w_{k-1}) - \eta \nabla
  \bar{G}^N_k(\bar{z}_k),
  \label{eq:ref_MHE}
\end{align}
where the above is derived similarly to the noiseless case.
Let~$\lbrace z_k \rbrace_{k \in \naturals}$ and~$\lbrace \bar{z}_k
\rbrace_{k \in \naturals}$ be the estimate sequences generated
by~\eqref{eq:MHE_rec_est} and~\eqref{eq:ref_MHE} respectively,
with~$z_0 = \bar{z}_0$, for which we have:
\begin{align*}
  \| z_k - \bar{z}_k \|  &= \| f_0(z_{k-1}) - f(\bar{z}_{k-1}, w_{k-1}) - \eta \nabla G^N_k(z_k)
  + \eta \nabla \bar{G}^N_k(\bar{z}_k)  \| \\
  & = \| f_0(z_{k-1}) - f_0(\bar{z}_{k-1})  + f_0(\bar{z}_{k-1}) - f(\bar{z}_{k-1}, w_{k-1}) \\
  & ~~~- \eta \nabla G^N_k(z_k) + \eta \nabla
  G^N_k(\bar{z}_k)
  - \eta \nabla G^N_k(\bar{z}_k) + \eta \nabla \bar{G}^N_k(\bar{z}_k)  \| \\
  &\leq c_f^{(1)} \| z_{k-1} - \bar{z}_{k-1} \| + c_f^{(2)} \| w_{k-1}
  \|
  +  \eta l \| z_k - \bar{z}_k \| \\
  &~~~ + l_w \eta \|
  (\mathbf{w}_{k:k+N-1},\mathbf{v}_{k+1:k+N}) \|,
\end{align*}
where the final inequality follows from
Assumptions~\ref{ass:system_lipschitz},~\ref{ass:cost_func_lipschitz}, and~\ref{ass:bar_G}, on the
several Lipschitz properties of $f$ the gradient of $G^N_k$, and
$\bar{G}^N_k$, respectively.  Further, since~$\eta l < 1$, we obtain
from the above that:
\begin{align*}
  \|  z_k - \bar{z}_k \| 
  & \leq \left( \frac{1}{1 - \eta l} \right) \left( c_f^{(1)} \|
    z_{k-1} - \bar{z}_{k-1} \|
    + c_f^{(2)} \| w_{k-1} \|  \right. \\
  & ~~~ \left. + 
    \eta l_w \| (\mathbf{w}_{k:k+N-1},\mathbf{v}_{k+1:k+N}) \| \right) \\
  & \leq \left( \frac{c_f^{(1)}}{1 - \eta l} \right)^k \| z_0 -
  \bar{z}_0 \|
  + \frac{c_f^{(2)}}{c_f^{(1)}} \sum_{\ell=1}^k \left( \frac{c_f^{(1)}}{1 - \eta l} \right)^\ell \| w_{k-\ell} \| \\
  & ~~~
  + \frac{\eta l_w}{c_f^{(1)}}  \sum_{\ell=1}^k \left( \frac{c_f^{(1)}}{1
      - \eta l} \right)^\ell  \hspace{-0.05in}
  \| (\mathbf{w}_{k-\ell+1:k-\ell+N},\mathbf{v}_{k-\ell+2:k-\ell+N+1}) \| \\
  & \leq \frac{c_f^{(2)}}{c_f^{(1)}} W C_k + \frac{\eta l_w \sqrt{N}
  }{c_f^{(1)}} (W + V) C_k .
\end{align*}
We note that if~$ \frac{c_f^{(1)}}{1 - \eta l} < 1$, we have
that~$\lim_{k \rightarrow \infty} C_k = \frac{c_f^{(1)}}{1 - \eta l -
  c_f^{(1)}}$ is finite, and therefore,~$\| z_k - \bar{z}_k \|$ is
bounded as $k \rightarrow \infty$. We note here that even when~$z_0
\neq \bar{z}_0$, the effect of this initial discrepancy vanishes as~$k
\rightarrow \infty$. 

Now, let~$T_k : \mathcal{K}_k \rightarrow \bar{\mathcal{K}}_k$ 
be a map such that for sequences~$\lbrace z_k \rbrace$ and~$\lbrace \bar{z}_k \rbrace$ 
generated by~\eqref{eq:MHE_rec_est} and~\eqref{eq:ref_MHE} respectively, 
with~$z_0 = \bar{z}_0$, we have~$T_k(z_k) = \bar{z}_k$.
It then follows that~${T_k}_{\#} \mu_k = \bar{\mu}_k$. 
Now, from the above, and by definition of the~$2$-Wasserstein distance,
we have: 
\begin{align*}
  W_2(\mu_k, \bar{\mu}_k) &\leq \left( \int_{z \in \mathcal{K}_k} \|z - T_k(z) \|^2 d\mu_k(z) \right)^{\frac{1}{2}} \\
  &\leq  \left( \int_{z \in \mathcal{K}_k} \left| \frac{c_f^{(2)}}{c_f^{(1)}} W C_k + \frac{\eta l_w \sqrt{N}}{c_f^{(1)}} (W + V) C_k \right|^2 d\mu_k(z) \right)^{\frac{1}{2}} \\
  &\leq \frac{c_f^{(2)}}{c_f^{(1)}} W C_k + \frac{\eta l_w
    \sqrt{N}}{c_f^{(1)}} (W + V) C_k.
\end{align*}
\end{proof}
\section{A KL-Moving-Horizon Estimator}
\label{sec:KL-MHE}

In this section, we derive a moving-horizon estimator, which we refer
to as~KL-MHE, to generate a sequence of probability
distributions~$\lbrace \mu_k \rbrace_{k \in \naturals}$. Using the
KL-divergence~$\subscr{D}{KL}$ as the choice of divergence in the
moving-horizon formulation~\eqref{eq:MHE_unifying_formulation}, we
obtain:
\begin{align}
  \begin{aligned}
    \mu_k \in \arg \min_{\mu \in \mathcal{P}(\Xsp)} &
    \subscr{D}{KL}(\mu \| f_{0\#} \mu_{k-1}) + \eta \mathbb{E}_{\mu} \left[ G^N_k \right], \\
    \text{given}~~\mu_0 \in \mathcal{P}(\Xsp)&.
  \end{aligned}
  \label{eq:MHE_KL-div}
\end{align}

%We represent the above in a
%compact form using the proximal operator on~$\mathcal{P}(\Xsp)$
%associated with $\mathbb{E}_\mu \left[ G^N_k \right] $ and
%w.r.t.~the KL-divergence, which we denote as follows:
%\begin{align}
%  \mu_k \in \text{prox}^{\subscr{D}{KL}}_{\eta G^N_k } \left(f_{0\#}  \mu_{k-1}
%  \right),  \quad k > 0.
%  \label{eq:MHE_prox_KL_map}
%\end{align}

%\subsection{Sampling scheme for KL-MHE}
%
We note that any local minimizer~$\mu_k$ of~\eqref{eq:MHE_KL-div} is
a critical point of the objective functional, and, therefore, it
satisfies:
\begin{align*}
	c = \frac{\delta}{\delta \mu} \left[   \subscr{D}{KL}(\mu \| f_{0\#} \mu_{k-1})
	+ \eta \mathbb{E}_{\mu} \left[ G^N_k \right]
   \right] \bigg|_{\mu = \mu_k},
\end{align*} 
where~$c$ is a constant
(from the constraint~$\int_{\Xsp} d\mu(x) = 1$, for $\mu \in
\mathcal{P}(\Xsp)$, due to which the first variation is defined up to
an additive constant). From the above, we get:
\begin{align*}
	c = \log \left( \frac{\rho_k}{ f_{0\#} \rho_{k-1}} \right) (x) + \eta G^N_k (x),
\end{align*}
where for any~$\ell \in \lbrace 0,1, \ldots \rbrace$,~$\rho_{\ell}$ is the density function corresponding to the
measure~$\mu_{\ell}$.
Therefore, the corresponding recursive update scheme for the density function is given by:
\begin{align}
  \rho_k (x) = c_k \left( f_{0\#} \rho_{k-1}(x)  \right)
  \exp \left( -  \eta G^N_k(x) \right),
  \label{eq:KL-MHE_rec_update}
\end{align}
where~$c_k$ is the normalization constant. 
We note that the above is a particle filter formulation,
with the horizon cost~$G^N_k$ defining the weighting function. Implementable filters are
obtained by a Sequential Monte Carlo method, see~\cite{AD-NDF-NG:01}.
We now present the asymptotic stability result for KL-MHE:
\begin{theorem}\longthmtitle{Asymptotic stability of KL-MHE}
  The estimator~\eqref{eq:MHE_KL-div}, under Assumptions~\ref{ass:system_lipschitz} 
  to~\ref{ass:lower_semicts_set_valued}, is an asymptotically
   stable observer for the system~$\Sigma$.
%   we have $\lim_{k  \rightarrow \infty} \mu_k \left( f^k_0 \left( {\Sigma_{T}}^{-1} (\yvec_{0:T}) \right) \right) =1$.
\end{theorem}
\begin{proof}
  We know that for any map~$\mathcal{T}$ and measure~$\mu$, we have
  that~$d \mathcal{T}_{\#} \mu (x) = d \mu \left( \mathcal{T}^{-1}(x)
  \right)$.  It then follows from~\eqref{eq:KL-MHE_rec_update} that:
\begin{align*}
  \rho_k (x) = c_k \rho_{k-1} (f^{-1}_0(x)) \exp \left( - \eta
    G^N_k(x) \right).
\end{align*}
We now rewrite the above as:
\begin{align*}
  \rho_k (f_0(x)) = c_k \rho_{k-1}(x) \exp \left( - \eta
    G^N_k(f_0(x)) \right).
\end{align*}
Repeating the above process~$k$ times, we obtain:
\begin{align*}
  \rho_k (f^k_0(x)) = C_k \rho_0 (x) \exp \left( - \eta \sum_{\ell=1}^k
    G^N_{\ell}(f^{\ell}_0(x)) \right),
\end{align*}
where~$C_k = c_k c_{k-1} \ldots c_1$ is the normalization constant.
If~$x \notin {\Sigma_T}^{-1} (\yvec_{0:T})$, we have that~$\lim_{k
  \rightarrow \infty} \rho_k (f^k_0(x)) = 0$, since~$\sum_{\ell=1}^k
G^N_{\ell}(f^{\ell}_0(x)) \rightarrow \infty$ as $k \rightarrow
\infty$ for all~$x \notin {\Sigma_{T}}^{-1}
(\yvec_{0:T})$ (by definition of the cost function, the sum diverges over
an infinitely long horizon). 
%
%\margin{can you explain further why the last part
%  holds? you are not using additivity of $J_T$ over each component,
%  right? or yes? }
%
 Thus, we get:
\begin{align*}
  \lim_{k \rightarrow \infty} \mu_k \left( f^k_0 \left(
      {\Sigma_{T}}^{-1} (\yvec_{0:T}) \right) \right) = \lim_{k
    \rightarrow \infty} \mu_k \left( {\Sigma_T}^{-1} (\yvec_{k:k+T})
  \right) = 1.
\end{align*}
\end{proof}
%
%\subsection{Robustness of KL-MHE}
%
\section{Differential privacy}
\label{sec:Diff_privacy}
% We begin, in this section, by adapting the notion of differential
% privacy, as defined in Section~\ref{sec:diff_privacy_prelim}, to the
% moving-horizon estimation setting.
In this section, we discuss the mechanism for encoding the desired
level of differential privacy in moving-horizon estimators. We then
apply this mechanism to the two estimators presented in the previous
sections, the~$W_2$-MHE and KL-MHE.  We conclude the section with a
discussion on differential privacy of the estimators over a time
horizon.
Our aim here is to guarantee differential privacy of the measurement data $\yvec_{0:T}$, 
when the estimate sequence $\lbrace \mu_k \rbrace$ is released (made public). 
We consider the class of scenarios where 
an adversary can access the released estimates, while the measurement data itself 
is not accessible to the adversary.
Our goal in incorporating differential privacy in estimation is to ensure that
the adversary is not able to distinguish (in the sense of $\epsilon$-differential privacy) 
between measurement sequences 
that are $\delta$-adjacent, using the released estimates, 
which is an underlying risk when the
estimates are directly released without such a consideration. 

Given the framework~\eqref{eq:MHE_unifying_formulation}, we encode
differential privacy by an entropic regularization of the estimation
objective function, as follows:
\begin{align}
  \begin{aligned}
    \mu_k \in \arg \min_{\mu \in \mathcal{P}(\Xsp)} & \left[ s_k D(\mu,
      f_{0\#} \mu_{k-1})
      + s_k \eta \mathbb{E}_{\mu} \left[ G^N_k \right] \right.\\
    & ~~~ \left. - (1-s_k) S^{\mathcal{K}_k}(\mu) \right], \\
    \text{given}~~\mu_0 \in \mathcal{P}(\Xsp),
  \end{aligned}
  \label{eq:MHE_gen_form_reg}
\end{align}
where~$s_k \in [0,1]$ is a tunable time-dependent parameter
and~$\mathcal{K}_k$ is the support of~$f_{0\#}\mu_{k-1}$
(with~$\mathcal{K}_0$ being the support of~$\mu_0$).
Moreover,~$S^A(\mu) = \int_{A} \rho \log(\rho) \dvol$, where~$A
\subset \Xsp$ and~$d\mu = \rho \dvol$. We note that
when~$s_k = 1$, the above formulation reduces
to~\eqref{eq:MHE_unifying_formulation} and when~$s_k = 0$, it is
equivalent to an entropy maximization problem, yielding a uniform
distribution over the set~$f_0(\mathcal{K}_{k-1})$ as the solution.
Clearly, the uniform distribution is insensitive to the measurements,
and therefore offers maximum privacy, while being of no value to the
estimation objective.  The ensuing analysis in this section is
directed at determining upper bounds on the parameter
sequence~$\lbrace s_k \rbrace_{k \in \naturals}$ such that the MHE
offers~$\epsilon$-differential privacy. We rewrite the optimization
problem~\eqref{eq:MHE_gen_form_reg} for~$s_k \in (0,1]$ as follows:
\begin{align}
  \begin{aligned}
    \mu_k \in \arg \min_{\mu \in \mathcal{P}(\Xsp)} & \left[ D(\mu,
      f_{0\#} \mu_{k-1})
      + \eta \mathbb{E}_{\mu} \left[ G^N_k \right] \right.\\
    & ~~~ \left. - \left( \frac{1-s_k}{s_k} \right) S^{\mathcal{K}_k}(\mu) \right], \\
    \text{given}~~\mu_0 \in \mathcal{P}(\Xsp),
  \end{aligned}
  \label{eq:MHE_gen_form_reg_rewrite}
\end{align}
%\margin{Instead of denoting $S^{\mathcal{K}_k}$, can we just use $S^k$ for
%  simplicity?}
% \marginv{I think it is better to retain it, since I have introduced the notation $S^A$ 
% for any set $A$. Calling it $S^k$ would require defining the shorthand again}
%
Let~$\yvec, \tilde{\yvec} \in \Ysp^{T+N+1}$ be two~$\delta$-adjacent
measurement sequences as in Definition~\ref{defn:diff_privacy}, over a
horizon~$\lbrace 0, \ldots, T+N \rbrace$, such that~$\| \yvec -
\tilde{\yvec} \| \leq \delta$ and let~$\lbrace \mu_k \rbrace_{k \in \naturals}$
and~$\lbrace \tilde{\mu}_k \rbrace_{k \in \naturals}$ be the sequences of estimates
derived from~\eqref{eq:MHE_gen_form_reg_rewrite}. In the following, we
determine conditions on $\lbrace s_k \rbrace_{k \in \naturals}$ that guarantee
differential privacy for each of the estimators derived in previous
sections.
\subsection{Differentially private~$W_2$-MHE}
We now design a differentially private~$W_2$-moving-horizon
estimator. We begin by considering:% the entropy-regularized $W_2$-MHE
% formulation, given by:
\begin{align}
  \begin{aligned}
    \mu_k \in \arg \min_{\mu \in \mathcal{P}(\Xsp)}
    &  \left[ \frac{1}{2} W_2^2 (\mu , f_{0\#} \mu_{k-1}) 
      + \eta \mathbb{E}_{\mu} \left[ G^N_k \right] \right. \\
    & ~~~ \left. - \left( \frac{1-s_k}{s_k} \right) S^{\mathcal{K}_k}(\mu) \right], \\
    \text{given}~~\mu_0 \in \mathcal{P}(\Xsp)&,
  \end{aligned}
  \label{eq:MHE_wasserstein_reg}
\end{align}
for~$s_k \in (0,1]$. 
%We propose a corresponding stochastic recursive estimator 
%for samples from~\eqref{eq:MHE_wasserstein_reg} as:
%\begin{align}
%  \begin{aligned}
%    &z_{k} = \arg \min_{z} \frac{1}{2} \left\| z - f_0(z_{k-1})
%    \right\|^2
%    + \eta G^N_k(z) \\ 
%    &\hspace{1in}+ \left( \frac{1-s_k}{s_k} \right) \frac{1}{2} \left\| z - \xi_k \right\|^2, \\
%    &z_0 \sim \mu_0 \in \mathcal{P}(\Xsp),
%  \end{aligned}
%  \label{eq:MHE_wasserstein_reg_opt}
%\end{align}
%where~$\xi_k \in \mathcal{K}_k$ is obtained by sampling uniformly from
%the set~$\mathcal{K}_k$. \margin{can we observe that we can always sample
%  uniformly from a superset containing $\mathcal{K}_k$? -- (as a
%  practical implementation question)}
%
%where~$\xi^k_{\eta}$ is generated by a reflected Brownian motion~$\xi^k_t$
%given by~$d\xi^k_t = \sqrt{2} \left( \frac{1-s_k}{s_k} \right) dB_t$
%over the duration~$\eta$ of the time step,
%starting at~$\text{prox}_{\eta G^N_k} (f_0(z_{k-1}))$.
%Note that when~$s_k = 1$, we get~$z_k = \xi^k_{\eta} = \text{prox}_{\eta G^N_k} (f_0(z_{k-1}))$.

The following theorem provides a sufficient upper
bound on~$s_T$ such that the entropy-regularized~$W_2$-MHE
in~\eqref{eq:MHE_wasserstein_reg} is~$\epsilon_T$-differentially
private at a time instant~$T$.
\begin{theorem}\longthmtitle{Sensitivity of~$W_2$-MHE}
  Given two~$\delta$-adjacent measurement sequences~$\yvec,
  \tilde{\yvec} \in \Ysp^{T+N+1}$, under Assumption~\ref{ass:cost_func_lipschitz},
   we have that the estimates generated by~\eqref{eq:MHE_wasserstein_reg} 
   satisfy~$\subscr{D}{max} \left( \mu_T, \tilde{\mu}_T \right) \leq \epsilon_T$ if~$s_T \leq
  \epsilon_T \left( \epsilon_T + c_f^T \text{diam}(\mathcal{K}_0)
    \left( \eta l \delta + c_f^T \text{diam}(\mathcal{K}_{0})
      q(\delta) \right) \right)^{-1}$, where~$q: \realnonnegative
  \rightarrow \realnonnegative$ is a class-$\mathcal{K}$ function that
  satisfies~$q(0) = 0$.
  \label{thm:sensitivity_W2_MHE}
\end{theorem}
\begin{proof}

  Let~$G^N_k$ and~${\widetilde{G}}^N_{k}$ be the estimation objective
  functions at time instant~$k$, corresponding to the measurement
  sequences~$\yvec$ and~$\widetilde{\yvec}$ respectively, and
  let~$\mu_k$ and~$\widetilde{\mu}_k$ be the respective estimated
  probability measures, with~$\rho_k, \widetilde{\rho}_k$ the
  corresponding density
  functions. From~\eqref{eq:MHE_wasserstein_reg}, we get that for
  all~$k \in \lbrace 0, \ldots, T \rbrace$,~$\mu_k$, being the local
  minimizer is also a critical point of the objective functional. We
  therefore obtain:
\begin{align*}
  \phi_k(x) + G^N_k(x) + \left( \frac{1-s_k}{s_k} \right) \log
  (\rho_k(x)) = c,
\end{align*}
where~$\phi_k$ is the Kantorovich potential associated with the transport from~$\mu_k$ to~${f_0}_{\#}\mu_{k-1}$ 
and~$c$ is a constant. It now follows that:
\begin{align*}
  \nabla \phi_k(x) + \nabla G^N_k(x) + \left( \frac{1-s_k}{s_k}
  \right) \nabla \log (\rho_k)(x) = 0.
\end{align*}
Similarly, we have:
\begin{align*}
  \nabla \widetilde{\phi}_k(x) + \nabla {\widetilde{G}}^N_{k}(x) +
  \left( \frac{1-s_k}{s_k} \right) \nabla \log (\widetilde{\rho}_k)(x)
  = 0.
\end{align*}
Taking the difference between the above two equations:
\begin{align*}
  \nabla \left[ \log \left( \frac{\rho_k}{\widetilde{\rho}_k} \right)
  \right] (x) = - \left( \frac{s_k}{1-s_k} \right) \left[
    \nabla (\phi_k - \widetilde{\phi}_k)(x)  \right. 
  \left. + \nabla (G^N_k - {\widetilde{G}}^N_{k})(x) \right].
\end{align*}
We have that~$\nabla \phi_k(x) = x - T_k^{-1}(x)$, where~$\mu_k =
{T_k}_{\#} \left( f_{0\#} \mu_{k-1} \right)$. This implies that~$
\nabla (\phi_k - \widetilde{\phi}_k)(x) = - (T_k^{-1}(x) -
\widetilde{T}_k^{-1}(x) )$.  However,~$T_k^{-1}(x),
\widetilde{T}_k^{-1}(x) \in f_0(\mathcal{K}_{k-1}) =
f^k_0(\mathcal{K}_0)$, and therefore~$\| \nabla (\phi_k -
\widetilde{\phi}_k)(x) \| \leq c_f^k \text{diam}(\mathcal{K}_0)
q(\delta)$, for all~$x \in f^k_0(\mathcal{K}_0)$ and some class-$\mathcal{K}$
function~$q$.  We let~$q$ characterize the dependence of~$\phi$
on the measurement sequence,
 and we get that $\| \nabla
(\phi_k - \widetilde{\phi}_k)(x) \| = 0$ for all~$x \in \Xsp$, when~$\delta = 0$.
Moreover, by Assumption~\ref{ass:cost_func_lipschitz}, we
get~$\|\nabla (G^N_k - {\widetilde{G}}^N_{k})(x)\| \leq l \delta$.
Therefore, we obtain:
\begin{align}
  \left\| \nabla \left[ \log \left( \frac{\rho_k}{\widetilde{\rho}_k}
      \right) \right] \right\| \leq \left( \frac{s_k}{1-s_k} \right)
  \left( c_f^k \text{diam}(\mathcal{K}_0) q(\delta) + l \delta
  \right).
  \label{eq:W2_grad_log_bound}
\end{align}
We also have that for any~$x \in f^k_0(\mathcal{K}_0)$:
\begin{align}
\begin{aligned}
  \log \left( \frac{\rho_k}{\widetilde{\rho}_k} \right) (x)
  =  \log \left( \frac{\rho_k}{\widetilde{\rho}_k} \right)(\bar{x}) 
  + \int_0^1 \nabla \left[ \log \left(
      \frac{\rho_k}{\widetilde{\rho}_k} \right) \right] (\gamma(t))
  \cdot \dot{\gamma}(t) dt,
  \end{aligned}
  \label{eq:W2_log_integral_curve}
\end{align}
where~$\gamma(0) = \bar{x}$ and~$\gamma(1) = x$. Since~$\rho_k$
and~$\widetilde{\rho}_k$ are continuous,
with~$\int_{f^k_0(\mathcal{K}_0)} (\rho_k - \widetilde{\rho}_k) = 0$
(since $\int_{f^k_0(\mathcal{K}_0)} \rho_k =
\int_{f^k_0(\mathcal{K}_0)} \widetilde{\rho}_k = 1$), there exists
an~$\bar{x} \in f^k_0(\mathcal{K}_0)$ such that $\rho_k(\bar{x}) =
\widetilde{\rho}_k(\bar{x})$, which implies that~$ \log \left(
  \frac{\rho_k}{\widetilde{\rho}_k} \right) (\bar{x}) = 0$.
From~\eqref{eq:W2_grad_log_bound}
and~\eqref{eq:W2_log_integral_curve}, for a straight line
segment~$\gamma$, we therefore obtain:
\begin{align*}
  \left| \log \left( \frac{\rho_k}{\widetilde{\rho}_k} \right)(x)
  \right| \leq \left( \frac{s_k}{1-s_k} \right)
  \left( c_f^k \text{diam}(\mathcal{K}_0) q(\delta) + l \delta \right) \times 
   c_f^k \text{diam}(\mathcal{K}_0),
\end{align*}
where we have used the fact that~$\int_{0}^1 |\dot{\gamma}(t)| dt = \|
x - \bar{x} \| \leq \text{diam}({f_0^k}(\mathcal{K}_0)) \leq c_f^k
\text{diam}(\mathcal{K}_0)$.  Thus, for~$k= T$, we let:
\begin{align*}
  \left| \log \left( \frac{\rho_T}{\widetilde{\rho}_T} \right) (x) \right|
  \leq \left( \frac{s_T}{1-s_T} \right) \left( c_f^T
    \text{diam}(\mathcal{K}_0) q(\delta) + l \delta \right) \times 
   c_f^T \text{diam}(\mathcal{K}_0) 
  \leq \epsilon_T,
\end{align*}
from which we obtain that:
\begin{align*}
  s_T \leq \frac{\epsilon_T }{\left( \epsilon_T + c_f^T
      \text{diam}(\mathcal{K}_0) \left( \eta l \delta + c_f^T
        \text{diam}(\mathcal{K}_{0}) q(\delta) \right) \right)},
\end{align*}
and since~$\left| \log \left( \frac{\rho_T}{\widetilde{\rho}_T}
  \right)(x) \right|  \leq \epsilon_T$ for all~$x \in
f^T_0(\mathcal{K}_0)$, we have that~$\sup_{x \in f^T_0(\mathcal{K}_0)}
\left| \log \left( \frac{\rho_T}{\widetilde{\rho}_T} \right) \right| =
\subscr{D}{max}(\mu_T , \widetilde{\mu}_T) \leq \epsilon_T$.
\end{proof}
%
%\subsection{Differential privacy over a time horizon}
%
%We now recall that the MHE~\eqref{eq:MHE_wasserstein_reg} generates a
%sequence~$\lbrace \mu_k \rbrace_{k \in \naturals}$ of probability
%measures over the state space~$\Xsp$, given a sequence~$\lbrace y_k
%\rbrace_{k \in \naturals}$ of measurements. 
%
As noted earlier, Theorem~\ref{thm:sensitivity_W2_MHE}
provides a sufficient upper bound on~$s_T$ for differential privacy of the
estimate at~$T$. 
The goal, however, is to guarantee the desired level of differential privacy over a time
horizon~$\lbrace 0, \ldots, T \rbrace$.  The key issue here is that
the recursive update scheme of the estimator introduces a dependence
between the estimates at different time instants. This essentially
means that imposing an upper bound on sensitivity for the marginal
distributions~$\mu_k$ individually, without regard to the dependence
between these distributions, may not be sufficient. 
%
%\margin{after reading this, it may not be clear what we guaranteed in previous
%  sections. Doesn't theorem 5 and theorem 6 guarantee privacy of data
%  over at time $T$? are these theorems ignoring these dependences? If
%  so, can we explictly say it earlier}
%\marginv{fixed the description}
% 
Therefore, to guarantee the desired level of differential privacy over the time horizon, we
must impose an upper bound on the sensitivity of the joint
distribution~$\sigma \in \mathcal{P}(\Xsp^{T+1})$, where the
estimates~$\mu_k$ are the marginals of~$\sigma$
over~$\Xsp$. % From the
% sensitivity bound on the joint distribution, we obtain bounds on
% sensitivity of the marginals which we then impose on the MHE via our
% proposed mechanism.

The following theorem provides a sufficient upper bound on~$\lbrace s_k \rbrace_{k=1}^T$ such that
the entropy-regularized~$W_2$-MHE in~\eqref{eq:MHE_wasserstein_reg}
is~$\epsilon$-differentially private over a time horizon~$\lbrace 0,
\ldots, T \rbrace$. 
\begin{theorem}\longthmtitle{Differentially private~$W_2$-MHE}
  Given two~$\delta$-adjacent measurement
  sequences~$\yvec, \widetilde{\yvec} \in \Ysp^{T+N+1}$, 
  under Assumption~\ref{ass:cost_func_lipschitz},
  we have that the estimates generated by~\eqref{eq:MHE_wasserstein_reg}
  satisfy~$\subscr{D}{max} \left( \sigma , \widetilde{\sigma} \right) \leq
  \epsilon$ if~$\sum_{k=1}^T \left( \frac{s_k}{1-s_k} \right) c_f^k
  \leq \frac{\epsilon}{ l \delta \text{diam}(\mathcal{K}_0)
  }$.
  \label{thm:diff_privacy_W2_MHE}
\end{theorem}
\begin{proof}
  Let~$G^N_k$ and~${\widetilde{G}}^N_{k}$ be the estimation objective
  functions at time instant~$k$, corresponding to the measurement
  sequences~$\yvec$ and~$\widetilde{\yvec}$ respectively, and
  let~$\sigma$ and~$\widetilde{\sigma}$ be the respective joint
  probability measures over the horizon~$\lbrace 0, \ldots, T
  \rbrace$. With a slight abuse of notation, we allow~$\sigma$
  and~$\widetilde{\sigma}$ to also denote the joint density
  function. We now have:
\begin{align*}
  \sigma(x_0, x_1, \ldots, x_T) = \rho_0(x_0) \sigma(x_1, \ldots, x_T | x_0) 
  = \rho_0(x_0) \rho_1(x_1|x_0) \rho_2(x_2|x_1) \ldots
  \rho_T(x_T|x_{T-1}),
\end{align*}
where~$\rho_k(x_k|x_{k-1})$ is the marginal density at~$x_k$ at time
instant~$k$, given that the distribution at time instant~$k-1$ is
concentrated at~$x_{k-1}$. Moreover, we note that
the~$W_2$-MHE~\eqref{eq:MHE_wasserstein_reg} yields a Markov process,
which allows us to express~$\rho_k(x_k | x_{k-1}, \ldots, x_0) =
\rho_k(x_k|x_{k-1})$.  Now,~$\rho_k(x_k|x_{k-1})$ is the density
corresponding to the measure obtained by the following:
\begin{align*}
	  \begin{aligned}
    \mu_k \in \arg \min_{\mu \in \mathcal{P}(\Xsp)}
      \left[ \frac{1}{2} W_2^2 (\mu , \partial_{f_0(x_{k-1})}) 
      + \eta \mathbb{E}_{\mu} \left[ G^N_k \right] \right.
     \left. - \left( \frac{1-s_k}{s_k} \right) S^{\mathcal{K}_k}(\mu) \right],
  \end{aligned}
\end{align*}
where~$\partial_{\xi}$ is the Dirac measure concentrated
at~$\xi$. From the above, we get that for all~$k \in \lbrace 0,
\ldots, T \rbrace$,~$\mu_k$, being the local minimizer is also a
critical point of the objective functional. 
Applying similar steps to those in the proof of
Theorem~\ref{thm:sensitivity_W2_MHE}, we obtain: 
\begin{align*}
  \left| \log \left( \frac{\rho_k}{\widetilde{\rho}_k}
    \right)(x|x_{k-1}) \right| \leq \left( \frac{s_k}{1-s_k} \right) l
  \delta c_f^k \text{diam}(\mathcal{K}_0).
\end{align*}
Now, we have:
\begin{align*}
  \left| \log \left( \frac{\sigma}{\widetilde{\sigma}} \right) (x_0,
    \ldots, x_T) \right| &\leq \sum_{k=1}^T \left| \log \left(
      \frac{\rho}
      {\widetilde{\rho}} \right) (x_k | x_{k-1})  \right| 
  \leq \sum_{k=1}^T \left( \frac{s_k}{1-s_k} \right) l \delta c_f^k
  \text{diam}(\mathcal{K}_0).
\end{align*}
By taking
\begin{align*}
  l \delta \text{diam}(\mathcal{K}_0) \sum_{k=1}^T \left(
    \frac{s_k}{1-s_k} \right) c_f^k \leq \epsilon,
\end{align*}
 we obtain the following inequality:
\begin{align*}
  \sum_{k=1}^T \left( \frac{s_k}{1-s_k} \right) c_f^k \leq
  \frac{\epsilon}{ l \delta \text{diam}(\mathcal{K}_0) },
\end{align*}
and that~$\subscr{D}{max}(\sigma , \widetilde{\sigma}) \leq \epsilon$.
%\oprocend
\end{proof}
We note that for a given~$\epsilon$, the upper bound on the
sequence~$\lbrace s_k \rbrace$ decreases with~$\delta$.  In other
words, guaranteeing $\epsilon$-differential privacy w.r.t. measurement
sequences that are farther apart requires the addition of more noise
and a greater loss in estimation accuracy. This is because the
weighting on the entropic regularization term in the estimation
objective increases when~$s_k$ is reduced.  The same is the case
when~$\epsilon$ is reduced for a given~$\delta$, which corresponds to
a more stringent privacy requirement.

\subsection{Differentially private~KL-MHE}
We now design a differentially private~KL-moving-horizon estimator. We
begin by considering the entropy-regularized~KL-MHE formulation,
given by:
\begin{align}
	\begin{aligned}
          \mu_k \in \arg \min_{\mu \in \mathcal{P}(\Xsp)} & \left[
            \subscr{D}{KL}(\mu \| f_{0\#} \mu_{k-1}) +
            \eta \mathbb{E}_{\mu} \left[ G^N_k \right] \right. \\
          & ~~~ \left. -
            \left( \frac{1 - s_k}{s_k} \right) S^{\mathcal{K}_k}(\mu) \right], \\
          \text{given}~~\mu_0 \in \mathcal{P}(\Xsp)&,
	\end{aligned}
	\label{eq:MHE_KL-div_reg}
\end{align}
for~$s_k \in (0,1]$. The corresponding recursive update scheme for~\eqref{eq:MHE_KL-div_reg}
is given by:
\begin{align}
  \rho_k (x) &= c_k \left( f_{0\#}\rho_{k-1} (x) \right)^{s_k} e^{- \eta s_k G^N_k(x)},
  \label{eq:MHE_KL-div_reg_rec_est}
\end{align}
which will be derived in the proof of Theorem~\ref{thm:sensitivity_KL_MHE} below.

The following theorem provides a sufficient upper
bound on~$s_k$ such that the entropy-regularized~KL-MHE 
in~\eqref{eq:MHE_KL-div_reg} is~$\epsilon_T$-differentially private at a time instant~$T$,
while ignoring the correlations between the estimates~$\mu_k$ across time.
\begin{theorem}\longthmtitle{Sensitivity of KL-MHE}
  Given two~$\delta$-adjacent measurement sequences~$\yvec,
  \widetilde{\yvec} \in \Ysp^{T+N+1}$, under
  Assumption~\ref{ass:cost_func_lipschitz}, we have that the estimates
  generated by~\eqref{eq:MHE_KL-div_reg} satisfy~$\subscr{D}{max} \left(
    \mu_T , \widetilde{\mu}_T \right) \leq \epsilon_T$
  if~$\sum_{k=1}^T \left( \prod_{i=k}^T s_i \right) \leq \epsilon_T
  \left(2\eta \max_{k \in \lbrace 0, \ldots, T \rbrace}\left( \alpha_k
      + l c_f^k \delta \text{diam}(\mathcal{K}_0)
    \right)\right)^{-1}$, where~$\alpha_k = \min_{\xi \in
    f^k_0(\mathcal{K}_0)} \left| \left( G^N_k - {\widetilde{G}}^N_{k}
    \right)(\xi) \right| $.
      \label{thm:sensitivity_KL_MHE}
\end{theorem}
\begin{proof}
  Let~$G^N_k$ and~${\widetilde{G}}^N_{k}$ be the estimation objective
  functions at time instant~$k$, corresponding to the measurement
  sequences~$\yvec$ and~$\widetilde{\yvec}$ respectively, and
  let~$\mu_k$ and~$\widetilde{\mu}_k$ be the respective estimated
  probability measures, with~$\rho_k, \widetilde{\rho}_k$  the corresponding
  density functions. From~\eqref{eq:MHE_KL-div_reg}, we get that
  for all~$k \in \lbrace 0, \ldots, T \rbrace$,~$\mu_k$, being the
  local minimizer is also a critical point of the objective
  functional. We therefore obtain: 
\begin{align*}
  \frac{\delta}{\delta \mu} & \left[ \subscr{D}{KL}(\mu \| f_{0\#}
    \mu_{k-1}) + \eta \mathbb{E}_{\mu} \left[ G^N_k \right]
    - \left( \frac{1 - s_k}{s_k} \right) S^{\mathcal{K}_k}(\mu)  \right] \bigg |_{\mu_k} \\
  &= \bar{c}_k,
\end{align*}
from which we derive that:
\begin{align*}
  \log \left( \frac{\rho_k}{f_{0\#}\rho_{k-1}} \right) (x) + \eta
  G^N_k(x) + \left( \frac{1 - s_k}{s_k} \right) \log \rho_k (x) =
  \bar{c}_k.
\end{align*}
The above equation can be rewritten as follows:
\begin{align*}
  \rho_k (x) = c_k \left( f_{0\#}\rho_{k-1} (x) \right)^{s_k} e^{- \eta s_k G^N_k(x)} 
  = c_k \left( \rho_{k-1} (f_0^{-1}(x)) \right)^{s_k} e^{- \eta s_k
    G^N_k(x)},
\end{align*}
where~$c_k$ is the normalization constant. We therefore obtain:
\begin{align*}
  \rho_k (f_0(x)) = c_k \left( \rho_{k-1} (x) \right)^{s_k} e^{- \eta
    s_k G^N_k(f_0(x))}.
\end{align*}
Expanding the above, we get:
\begin{align*}
  \rho_T (f_0^T(x)) = C_T \left( \rho_0(x) \right)^{\prod_{k=1}^T s_k}
  e^{- \eta \sum_{k=1}^T \left( \prod_{i=k}^T s_i \right)
    G^N_k(f_0^k(x))},
\end{align*}
where~$C_T = c_1 c_2 \ldots c_T$. Similarly, we have:
\begin{align*}
  \widetilde{\rho}_T (f_0^T(x)) = \widetilde{C}_T \left(
    \widetilde{\rho}_0(x) \right)^{\prod_{k=1}^T s_k} e^{- \eta
    \sum_{k=1}^T \left( \prod_{i=k}^T s_i \right)
    {\widetilde{G}}^N_{k}(f_0^k(x))},
\end{align*}
where~$\widetilde{C}_T = \widetilde{c}_1 \widetilde{c}_2 \ldots
\widetilde{c}_T$
%\margin{same comment about the $\tilde{C}_T$}
%\marginv{ $C_T \left( \rho_0(x) \right)^{\prod_{k=1}^T s_k}$ is just
% a product not a function C_T of the term in the paranthesis}
and~$\rho_0 = \widetilde{\rho}_0$, as we assume that the estimator
starts with the same initial~$\mu_0$.  From the above two equations,
we obtain:
\begin{align*}
  \log \left( \frac{\rho_T}
    {\widetilde{\rho}_T} \right) (f_0^T(x)) = \log \left( \frac{C_T}{\widetilde{C}_T} \right) 
   - \eta \sum_{k=1}^T \left( \prod_{i=k}^T s_i
  \right) \left( G^N_k - {\widetilde{G}}^N_{k} \right) (f_0^{k}(x)).
\end{align*}
%It follows from the above that:
%\begin{align*}
%	&\nabla \left[ \log \left( \frac{\rho_T}{\widetilde{\rho}_T} \right) \right] (x) \\
%					&=  - \eta \sum_{k=1}^T \left( \prod_{i=k}^T s_i \right) \nabla \left( G^N_k - {\widetilde{G}}^N_{k} \right) (f_0^{-(T-k)}(x)) \cdot \\
%								&\hspace*{1.5in}	\nabla  (f_0^{-(T-k)})(x)
%\end{align*}	
%From Assumption~\eqref{ass:cost_func_lipschitz}, we have that~$\left| \nabla \left( G^N_k - {\widetilde{G}}^N_{k} \right)(f_0^{-(T-k)}(x)) \right|  \leq l \delta$,
%and we get: 
%\begin{align*}
%	\left|  \nabla \left[ \log \left( \frac{\rho_T}{\widetilde{\rho}_T} \right) \right] (x) \right| \leq \eta l \delta \sum_{k=1}^T \left( \prod_{i=k}^T s_i \right) \| \nabla f (x) \|^{-(T-k)}
%\end{align*}
%\begin{align*}
%	\left( p_K- \widetilde{p}_K \right) (x) = \left( p_K- \widetilde{p}_K \right) (x_0) + \int \nabla \left( p_K- \widetilde{p}_K \right) \cdot \mathbf{r}
%\end{align*}
%We get:
%\begin{align*}
%	\left | \left( p_K- \widetilde{p}_K \right) (x) \right | \leq  \eta \| x - x_0 \| \left( \sum_{k=0}^K w^{K-k+1} \right) \delta_K
%\end{align*}
%We have that~$x , x_0 \in f_0^K (\mathcal{K}_0)$ above.
%With the choice~$\left( \sum_{k=1}^{K+1} w^{k} \right) \leq \frac{\epsilon_K}{\eta \max \lbrace c_f , c_f_0^{K} \rbrace \text{diam}(\mathcal{K}_0) \delta_K}$, 
%we get that~$D_{\max} \left( \mu_K || \widetilde{\mu}_K \right) \leq \epsilon_K$.
%
The max-divergence between~$\mu_T$ and~$\widetilde{\mu}_T$ can be
upper bounded now by:
\begin{align*}
  \subscr{D}{max}(\mu_T , \widetilde{\mu}_T) &=
  \sup_{x \in \mathcal{K}_0} \left| \log \left( \frac{\rho_T}{\widetilde{\rho}_T} \right) (f_0^T(x)) \right| \\
  &\leq  \left| \log \left( \frac{C_T}{\widetilde{C}_T} \right)  \right| + \sup_{x \in \mathcal{K}_0} \eta \sum_{k=1}^T \left( \prod_{i=k}^T s_i \right) \times 
   \left| \left( G^N_k - {\widetilde{G}}^N_{k} \right) (f_0^{k}(x)) \right| \\
  &\leq 2 \sup_{x \in \mathcal{K}_0} \eta \sum_{k=1}^T \left(
    \prod_{i=k}^T s_i \right) \left| \left( G^N_k -
      {\widetilde{G}}^N_{k} \right) (f_0^{k}(x)) \right|,
\end{align*}
where the final inequality is due to the following (note that we use the fact that~$\rho = \tilde{\rho}$, as
mentioned earlier):
\begin{align*}
   \left| \log \left( \frac{C_T}{\widetilde{C}_T} \right)  \right|  
  &= \left| \log \left( \frac{\int_{x \in \mathcal{K}_0} \left(
          \rho_0(x) \right)^{\prod_{k=1}^T s_k} e^{- \eta \sum_{k=1}^T
          \left( \prod_{i=k}^T s_i \right) G^N_k(f_0^k(x))}}{\int_{x
          \in \mathcal{K}_0} \left( \widetilde{\rho}_0(x)
        \right)^{\prod_{k=1}^T s_k}
        e^{- \eta \sum_{k=1}^T \left( \prod_{i=k}^T s_i \right) {\widetilde{G}}^N_{k}(f_0^k(x))}} \right)  \right| \\
  & \leq \sup_{x \in \mathcal{K}_0} \left| \log \left( \frac{e^{- \eta
          \sum_{k=1}^T \left( \prod_{i=k}^T s_i \right)
          G^N_k(f_0^k(x))}}
      {e^{- \eta \sum_{k=1}^T \left( \prod_{i=k}^T s_i \right) {\widetilde{G}}^N_{k}(f_0^k(x))}} \right)  \right| \\
  &\leq \sup_{x \in \mathcal{K}_0} \eta \sum_{k=1}^T \left(
    \prod_{i=k}^T s_i \right) \left| \left( G^N_k -
      {\widetilde{G}}^N_{k} \right) (f_0^{k}(x)) \right|.
\end{align*}
%
%\margin{explain the second inequality: first, I think we already
%  discussed the supremum, and I believed it, but it would be good you
%  wrote it in a little lemma in the appendix. Second, why do you
%  remove the quotient of $\rho $ and $\tilde{\rho}$?}
%  \marginv{because $\rho = \tilde{rho}$}
%
We now have, for all~$k \in \lbrace 1, \ldots, T \rbrace$:
\begin{align*}
  \left( G^N_k - {\widetilde{G}}^N_{k} \right) (f_0^{k}(x)) =
  \left( G^N_k - {\widetilde{G}}^N_{k} \right) (\xi_k) 
   + \int_0^1 \nabla \left( G^N_k -
    {\widetilde{G}}^N_{k} \right) (\gamma_k(t)) \cdot
  \dot{\gamma}_k(t) dt,
\end{align*}
where~$\gamma_k(0) = \xi_k$ and~$\gamma_k(1) = f_0^k(x)$. From
Assumption~\ref{ass:cost_func_lipschitz}, we have~$\left\| \nabla
  \left( G^N_k - {\widetilde{G}}^N_{k} \right) (\xi) \right\| \leq l
\delta$.  Moreover, let~$\xi_k \in f_0^k(\mathcal{K}_0)$ such
that~$\left| \left( G^N_k - {\widetilde{G}}^N_{k} \right) (\xi_k)
\right| = \min_{f_0^k(\mathcal{K}_0)} \left| \left( G^N_k -
    {\widetilde{G}}^N_{k} \right) \right| = \alpha_k$, and we obtain:
\begin{align*}
  \left| \left( G^N_k - {\widetilde{G}}^N_{k} \right) (f_0^{k}(x))
  \right|
  \leq \alpha_k + l \delta \text{diam}(f_0^k(\mathcal{K}_0)) 
  \leq \alpha_k + l c_f^k \delta \text{diam}(\mathcal{K}_0).
\end{align*}
This yields the following inequality:
\begin{align*}
  2 \sup_{x \in \mathcal{K}_0} \eta \sum_{k=1}^T
  \left( \prod_{i=k}^T s_i \right) \left| \left( G^N_k - {\widetilde{G}}^N_{k} \right) (f_0^{k}(x)) \right| 
  &\leq 2 \eta \sum_{k=1}^T
  \left( \prod_{i=k}^T s_i \right) \left( \alpha_k + l c_f^k \delta \text{diam}(\mathcal{K}_0) \right) \\
  &\leq 2 \eta \max_{k}\left( \alpha_k + l c_f^k \delta
    \text{diam}(\mathcal{K}_0) \right) \sum_{k=1}^T \left(
    \prod_{i=k}^T s_i \right).
\end{align*}
We now let:
\begin{align*}
  2 \eta \max_{k}\left( \alpha_k + l c_f^k \delta
    \text{diam}(\mathcal{K}_0) \right) \sum_{k=1}^T \left(
    \prod_{i=k}^T s_i \right) \leq \epsilon_T,
\end{align*}
which yields the bound
\begin{align*}
  \sum_{k=1}^T \left( \prod_{i=k}^T s_i \right) \leq
  \frac{\epsilon_T}{2\eta \max_{k}\left( \alpha_k + l c_f^k \delta
      \text{diam}(\mathcal{K}_0) \right)},
\end{align*}
and we get~$\subscr{D}{max} (\mu_T , \widetilde{\mu}_T) \leq \epsilon_T$.
%
%\margin{any ways in which we can compute or bound $\alpha_k$?
%  Otherwise, I'm sure reviewers will complain about this}
%\margin{how about if we try to calculate the bounds in an example at
%  least? if data are bounded, maybe it's not so hard}
%\oprocend
\end{proof}
We note here that, in practice, with the choice of a sufficiently
large domain~$\mathcal{K}_0$, we can ensure that~$\alpha_k = \min_{\xi
  \in f^k_0(\mathcal{K}_0)} \left| \left( G^N_k -
    {\widetilde{G}}^N_{k} \right)(\xi) \right| = 0$ for all~$k \in
\lbrace 0, \ldots, T \rbrace$.  This is owing to the fact that for a
large enough~$\mathcal{K}_0$, we will have~$\min_{\xi \in
  f^k_0(\mathcal{K}_0)} \left( G^N_k - {\widetilde{G}}^N_{k}
\right)(\xi) \leq 0 \leq \max_{\xi \in f^k_0(\mathcal{K}_0)} \left(
  G^N_k - {\widetilde{G}}^N_{k} \right)(\xi)$. Moreover, since the
function~$G^N_k - {\widetilde{G}}^N_{k}$ is continuous, there must
therefore exist a point~$\xi^*$ such that $\left( G^N_k -
  {\widetilde{G}}^N_{k} \right) (\xi^*) = 0$.
  
As with the~$W_2$-MHE, we now characterize the differential privacy of
the KL-MHE over a horizon~$\lbrace 0, \ldots, T \rbrace$. We recall
that the KL-MHE yields a sequence of distributions~$\lbrace \mu_k
\rbrace_{k=0}^T$ over the time horizon.  Differential privacy over the
horizon requires an upper bound on the sensitivity of the joint
distribution~$\sigma$ over the horizon, where~$\mu_k$ is the marginal
of~$\sigma$ at the time instant~$k$.  As before, with a slight abuse
of notation, letting~$\sigma$ also denote the joint density function,
we have:
\begin{align*}
	  \sigma(x_0, x_1, \ldots, x_T) = \rho_0(x_0) \sigma(x_1, \ldots, x_T | x_0)
  = \rho_0(x_0) \rho_1(x_1|x_0) \rho_2(x_2|x_1) \ldots
  \rho_T(x_T|x_{T-1}).
\end{align*}
From the above, we infer that to estimate the sensitivity of the joint
density function, we must estimate the sensitivity of the
conditionals~$\rho_k(x_k | x_{k-1})$.  The conditional~$\rho_k(x_k |
x_{k-1})$ at any time instant~$k$, is obtained from the coupling
between the marginal distributions~$\mu_k$ and~$\mu_{k-1}$.
%
%For a sufficient upper bound on~$s_k$ to guarantee differential privacy over the time horizon, 
%we must ideally consider the
%coupling with maximum sensitivity to measurements.

We now obtain an upper bound for the case where the marginals~$\mu_k$ are
independently coupled. In other
words, we suppose that:
\begin{align}
\begin{aligned}
	  \sigma(x_0, x_1, \ldots, x_T) = \rho_0(x_0) \sigma(x_1, \ldots, x_T | x_0) 
  = \rho_0(x_0) \rho_1(x_1) \rho_2(x_2) \ldots
  \rho_T(x_T).
\end{aligned}
\label{eq:KL_independent_coupling}
\end{align}
\begin{theorem}\longthmtitle{Differentially private KL-MHE}
  Given two~$\delta$-adjacent measurement sequences~$\yvec,
  \widetilde{\yvec} \in \Ysp^{T+N+1}$, under
  Assumption~\ref{ass:cost_func_lipschitz} and the independent
  coupling~\eqref{eq:KL_independent_coupling}, we have that the
  estimates generated by~\eqref{eq:MHE_KL-div_reg}
  satisfy~$\subscr{D}{max} \left( \sigma , \widetilde{\sigma} \right)
  \leq \epsilon$ if~$ \sum_{k=1}^T \sum_{l=1}^k \left(\prod_{i=l}^k
    s_i \right) \leq \epsilon \left( 2 \eta \max_{k}\left( \alpha_k +
      l c_f^k \delta \text{diam}(\mathcal{K}_0) \right)
  \right)^{-1}$, where~$\alpha_k = \min_{\xi \in
    f^k_0(\mathcal{K}_0)} \left| \left( G^N_k - {\widetilde{G}}^N_{k}
    \right)(\xi) \right| $.
    \label{thm:diff_privacy_KL_MHE}
\end{theorem}
\begin{proof}
	  Let~$G^N_k$ and~${\widetilde{G}}^N_{k}$ be the estimation objective
  functions at time instant~$k$, corresponding to the measurement
  sequences~$\yvec$ and~$\widetilde{\yvec}$ respectively, and
  let~$\sigma$ and~$\widetilde{\sigma}$ be the respective joint
  probability measures over the horizon~$\lbrace 0, \ldots, T
  \rbrace$. With a slight abuse of notation, we allow~$\sigma$
  and~$\widetilde{\sigma}$ to also denote the joint density
  function. From~\eqref{eq:KL_independent_coupling}, we get:
  \begin{align*}
    \log \left( \frac{\sigma}{\widetilde{\sigma}} \right) (x_0,
    \ldots, x_T) = \sum_{k=1}^T \log \left(
      \frac{\rho_k}{\widetilde{\rho}_k} \right)(x_k),
  \end{align*}
  which implies that:
  \begin{align*}
    \subscr{D}{max}(\sigma, \widetilde{\sigma}) \leq \sum_{k=1}^T
    \subscr{D}{max}(\mu_k, \widetilde{\mu}_k).
  \end{align*}
  From the proof of Theorem~\ref{thm:sensitivity_KL_MHE} on the sensitivity of KL-MHE,
  we further get:
  \begin{align*}
    \subscr{D}{max}(\sigma, \widetilde{\sigma}) \leq \sum_{k=1}^T \subscr{D}{max}(\mu_k, \widetilde{\mu}_k) 
     \leq 2 \eta \max_{k}\left( \alpha_k + l c_f^k \delta
      \text{diam}(\mathcal{K}_0) \right) \sum_{k=1}^T \sum_{l=1}^k
    \left( \prod_{i=l}^k s_i \right).
  \end{align*}
  Therefore, it holds that~$\subscr{D}{max}(\sigma, \widetilde{\sigma}) \leq \epsilon$ if:
  \begin{align*}
  		\sum_{k=1}^T \sum_{l=1}^k \left(\prod_{i=l}^k s_i \right) \leq \frac{\epsilon}{	2 \eta \max_{k}\left( \alpha_k + l c_f^k \delta
    \text{diam}(\mathcal{K}_0) \right)}.
  \end{align*}
\end{proof}
\section{Simulation results}
\label{sec:numerical_expt}
In this section, we present results from numerical simulations
of the estimators studied in this paper. 
The simulations were performed in MATLAB (version R2017a) on 
a~2.5 GHz Intel Core i5 processor.

We considered the following nonlinear discrete-time system:
\begin{align*}
	x_1(k+1) &= x_1(k) + \tau x_2(k), \\
	x_2(k+1) &= x_2(k) - \tau \frac{x_1(k)}{1 + |x_1(k)|^2 + |x_2(k)|^2} + w_k, \\
	y(k) &= x_1(k) + v_k,
\end{align*}
with~$\tau = 0.1$,~$w_k$ and~$v_k$ are i.i.d disturbances, sampled uniformly 
from the intervals~$[-0.1, 0.1]$ and~$[-0.15, 0.15]$ respectively,
and a quadratic estimation objective function $J_T (\mathbf{y}^{(1)}_{0:T}, \mathbf{y}^{(2)}_{0:T})
 = \| \mathbf{y}^1_{0:T} - \mathbf{y}^2_{0:T} \|^2$.

We first present the simulation results for~$W_2$-MHE.  We ran
30~trials of the estimator~\eqref{eq:MHE_opt} on the same measurement
sequence, with randomly generated initial conditions and over a time
horizon of length~$T = 100$.  The length of the moving-horizon was
chosen to be~$N = 10$.  Figure~\ref{fig:W2-MHE} contains the plots of
the mean of the estimates along with the true states. The root mean
squared error (RMSE) for the mean state estimate sequences were found
to be ${z_1}^{\RMSE} = 0.0856$ and~${z_2}^{\RMSE} = 0.0846$ for the
estimates of~$x_1$ and~$x_2$, respectively. The average time for
computing the state estimate through the
minimization~\eqref{eq:MHE_opt} using the~$fminunc$ function in MATLAB 
was observed to be~$\subscr{t}{comp} =
0.012 \pm 0.02s$. 
%\margin{do you mean solving equation (9)? what method did you use to do this?}
%
\begin{figure}[!h]
	\begin{center}
          \hspace{-0.15in} \includegraphics[width=0.65\textwidth]{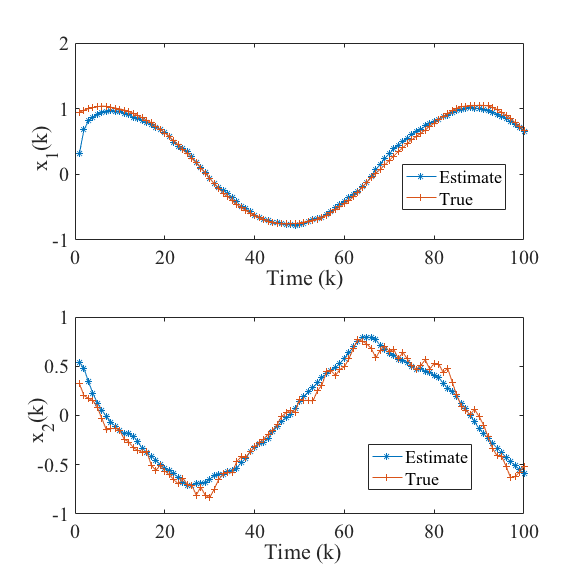}
          \captionsetup{justification=centering}
	\end{center}
   \caption{Mean state estimates from 30 trials of~$W2$-MHE}
	\label{fig:W2-MHE}
\end{figure}
%\margin{the fonts of the figures are quite bad (...)}
%\marginv{changed to Times New Roman}
%
%\begin{figure}[!h]
%	\begin{center}
%          \includegraphics[width=0.5\textwidth]{W2MHE_var.png}
%          \captionsetup{justification=centering}
%	\end{center}
%   \caption{Sample variance from Monte Carlo simulation of~$W2$-MHE with 30 samples}
%	\label{fig:W2-MHE_var}
%\end{figure}

We then implemented the estimator~\eqref{eq:MHE_KL-div} with~30
samples, over a time horizon of length~$T = 100$.  The length of the
moving-horizon was chosen to be~$N = 10$.  Figure~\ref{fig:KL-MHE}
contains the plots of the mean of the estimates along with the true
states.  The root mean squared error (RMSE) for the mean state
estimate sequences were found to be ${z_1}^{\RMSE} = 0.1073$
and~${z_2}^{\RMSE} = 0.1144$ for the estimates of~$x_1$ and~$x_2$,
respectively.  The average run-time for the
minimization~\eqref{eq:MHE_KL-div} by a resampling method was observed to
be~$\subscr{t}{comp} = ( 4.8 \pm 0.4 ) \times 10^{-4} s$.  

In simulation, with~30 samples, we find that the~$W_2$-MHE performs
better with respect to the root mean squared error, while the KL-MHE is much faster.
The performance of the KL-MHE is determined by the richness of the sample
set and effectiveness of the resampling procedure, choices that depend
on context and experience. In this manuscript, we did not attempt to
investigate improvements in performance with respect to these
choices. The performance of~$W_2$-MHE does not necessarily improve with the 
richness of the sample set, but for systems for which~${\Sigma_T}^{-1}(\yvec_{0:T})$
is not a singleton, a richer sample set allows for a more complete characterization of the
set of feasible estimates.

\begin{figure}[!h]
	\begin{center}
           \includegraphics[width=0.65\textwidth]{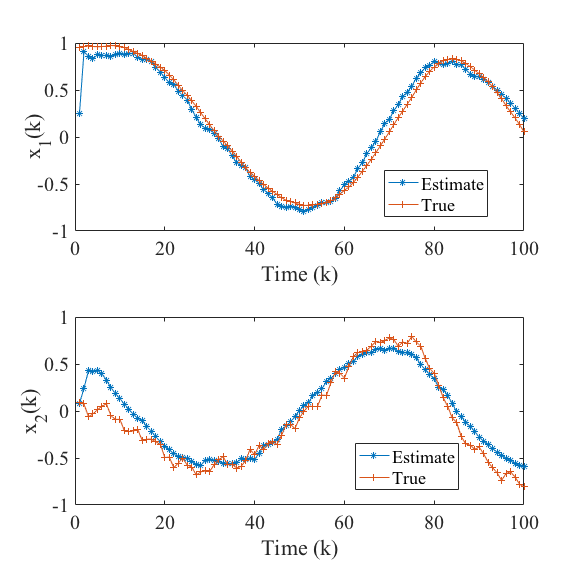}
          \captionsetup{justification=centering}
	\end{center}
   \caption{Mean state estimates from~KL-MHE with 30 samples}
	\label{fig:KL-MHE}
\end{figure}
\begin{figure}[!h]
	\begin{center}
          \includegraphics[width=0.65\textwidth]{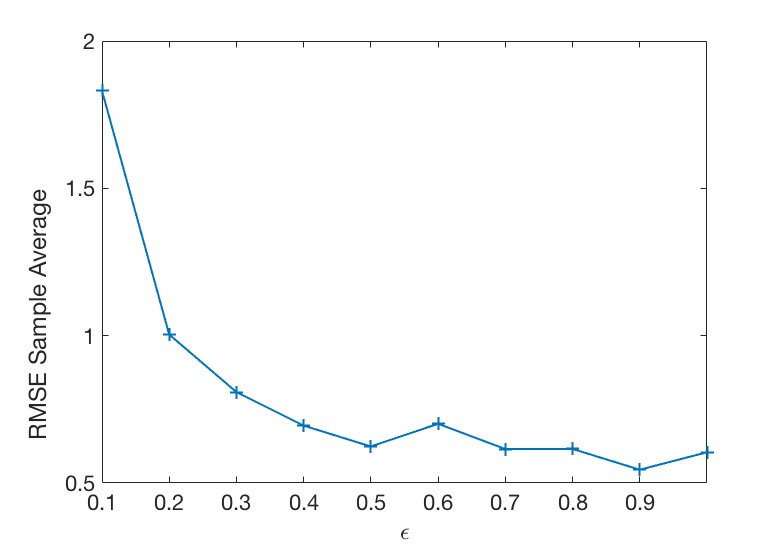}
          \captionsetup{justification=centering}
	\end{center}
   \caption{RMSE in estimates of state~$x_1$ for $W_2$-MHE, averaged over 30 samples for different values of~$\epsilon$}
	\label{fig:accuracy_privacy_tradeoff}
\end{figure}
Figure~\ref{fig:accuracy_privacy_tradeoff} illustrates the typical
trade-off between accuracy and privacy in moving-horizon estimation.
We considered constant weights~$s_k = s$ for the entropic
regularization terms in~\eqref{eq:MHE_wasserstein_reg}
and~\eqref{eq:MHE_KL-div_reg}. The values of~$s$ were chosen such that
they satisfied the bounds specified in
Theorems~\ref{thm:diff_privacy_W2_MHE}
and~\ref{thm:diff_privacy_KL_MHE} for $\epsilon$-differential privacy
of the estimators over the horizon. 
%
%\margin{In these bounds there are
%  a lot of constants, plus that scaling function. Can you indicate how
%  you chose those? Also, what is the associated $ \delta$? Can we plot
%  a function of the epsilon required for a given $\delta$?}
%\margin{Even if this doesn't make it to the manuscript, can we show in
%  several plots a typical $\delta$ adjacent data and the performance
%  of the corresponding differentially private filter on tracking a
%  trajectory?} 
%
 In Figure~\ref{fig:accuracy_privacy_tradeoff}, we plot
the RMSE (for the estimates of the state~$x_1$) for~$W_2$-MHE,
averaged over the~$30$ samples, specifying the accuracy, for different
values of~$\epsilon$, the privacy parameter.  We recall that a higher
value of~$\epsilon$ indicates a less stringent privacy requirement.
We notice that the the accuracy of the estimators improves with an
increase in the privacy parameter. 
%
%\margin{Can you show in a plot of
%  how the trajectory estimate worsens for a few choices of epsilon
%  (say three). And how about for the KL differentially private
%  version? We should have these plots even if we don't include them in
%  the manuscript. }
% \marginv{the trajectory estimate plots get really messy, so I have left them out}
%
\section{Conclusions}
\label{sec:conclusions}
In this work, we laid out a unifying probabilistic framework for
moving-horizon estimation.  We clearly established the connection
between the classical notion of strong local observability and the
stability of moving-horizon estimation, for nonlinear discrete-time
systems. We then proposed a differentially private mechanism 
based on entropic regularization and derived conditions under 
which $\epsilon$-differential privacy is guaranteed at any given time
instant and over time horizons.
As an extension to this work, we intend to include distributional
constraints in the moving-horizon estimation framework.
An important consideration in the estimation problem, in addition to
the asymptotic stability, is the rate of convergence of the
observer. It is of interest to obtain convergence rate bounds for the
moving-horizon estimators proposed in this paper, and to compare their
performance for various choices of the metric (or divergence) in the
unifying formulation, which will be undertaken in our future work.
\bibliographystyle{plain}
\bibliography{alias,SMD-add,SM,JC}

\end{document}